\documentclass[a4paper]{amsart}
\usepackage[utf8]{inputenc}
\usepackage[british]{babel}
\usepackage{csquotes}
\usepackage[useregional]{datetime2}
\usepackage[shortcuts]{extdash}
\usepackage[style=alphabetic, maxbibnames=10, maxalphanames=4]{biblatex}

\usepackage{amsmath}
\usepackage{amsfonts}
\usepackage{amsthm}
\usepackage{amssymb}
\usepackage{mathrsfs}
\usepackage{mathtools}

\usepackage[shortlabels]{enumitem}
\setlist[enumerate]{label=(\alph*)}

\usepackage[ruled,noline,linesnumbered]{algorithm2e}
\SetKwInOut{Input}{Input}
\SetKwInOut{Output}{Output}
\DontPrintSemicolon
\SetAlgoCaptionSeparator{.}
\counterwithin{algocf}{section}

\usepackage{breakurl}
\usepackage[hidelinks, pdfauthor={Johannes Schmitt}]{hyperref}
\usepackage[group-separator={,}]{siunitx}

\theoremstyle{plain}
\newtheorem{theorem}{Theorem}[section]
\newtheorem{lemma}[theorem]{Lemma}
\newtheorem{proposition}[theorem]{Proposition}

\theoremstyle{definition}
\newtheorem{definition}[theorem]{Definition}

\theoremstyle{remark}
\newtheorem{remark}[theorem]{Remark}
\newtheorem{notation}[theorem]{Notation}
\newtheorem{example}[theorem]{Example}

\makeatletter
\patchcmd{\@setauthors}{\footnotesize}{}{}{}
\patchcmd{\@setauthors}{\MakeUppercase}{}{}{}
\patchcmd{\@setaddresses}{\scshape}{}{}{}
\patchcmd{\contentsnamefont}{\scshape}{\bfseries}{}{}
\patchcmd{\abstract}{\scshape}{\bfseries}{}{}
\patchcmd{\section}{\scshape}{\bfseries}{}{}
\patchcmd{\@secnumfont}{\mdseries}{\bfseries}{}{}
\patchcmd{\@captionheadfont}{\scshape}{\bfseries}{}{}
\makeatother

\DeclareMathOperator{\Ab}{Ab}
\DeclareMathOperator{\Cl}{Cl}

\DeclareMathOperator{\GL}{GL}
\DeclareMathOperator{\gr}{gr}
\DeclareMathOperator{\Hom}{Hom}
\DeclareMathOperator{\im}{im}
\DeclareMathOperator{\initial}{in}
\DeclareMathOperator{\LM}{LM}
\DeclareMathOperator{\LT}{LT}

\DeclareMathOperator{\mindeg}{mindeg}
\DeclareMathOperator{\Spec}{Spec}
\DeclareMathOperator{\Supp}{Supp}
\DeclareMathOperator{\SL}{SL}

\newcommand{\CC}{\mathbb C}
\newcommand{\QQ}{\mathbb Q}
\newcommand{\ZZ}{\mathbb Z}

\renewcommand{\epsilon}{\varepsilon}
\renewcommand{\phi}{\varphi}
\renewcommand{\theta}{\vartheta}
\newcommand{\ideal}{\trianglelefteq}
\renewcommand{\hom}{\mathrm{hom}}
\newcommand{\deh}{\mathrm{deh}}

\title{Homogeneous Khovanskii bases and MUVAK bases}

\author{Johannes Schmitt}
\address{Johannes Schmitt, Ruhr\-/Universität Bochum, Fakultät für Mathematik, Universitätsstraße 150, 44801 Bochum, Germany}
\email{johannes.schmitt@ruhr-uni-bochum.de}
\date{\DTMdisplaydate{2025}{12}{02}{-1}}

\begin{document}

\begin{abstract}
  In 2019, Kaveh and Manon introduced Khovanskii bases as a special `Gröbner\-/like' generating system of an algebra.
  We extend their work by considering an arbitrary grading on the algebra and propose a definition for a `homogeneous Khovanskii basis' that respects this grading.
  We generalize Khovanskii bases further by taking multiple valuations into account (MUVAK bases).
  We give algorithms in both cases.

  MUVAK bases appear in the computation of the Cox ring of a minimal model of a quotient singularity.
  Our algorithm is an improvement of an algorithm by Yamagishi in this situation.
\end{abstract}

\maketitle
\tableofcontents

\section{Introduction}

\emph{SAGBI bases} -- the subalgebra analogue to Gröbner bases for ideals -- were introduced by Robbiano and Sweedler in \cite{RS90}.
One is given a subalgebra \(A\leq K[X_1,\dots,X_n]\) of a polynomial ring over a field \(K\) and a monomial ordering \(>\).
The algorithmic task is to find generators \(f_1,\dots, f_k\) of \(A\) with \[K[\LM_>(f_i)\mid 1\leq i\leq k] = K[\LM_>(f)\mid f\in A],\] where \(\LM_>(f)\) denotes the leading monomial of a polynomial \(f\) with respect to \(>\).
SAGBI bases are also called \emph{canonical bases} in \cite{Stu96}.
Kaveh and Manon \cite{KM19} generalized this concept by far with \emph{Khovanskii bases}.
In this setting, \(A\) is a finitely generated \(K\)\=/algebra without zero\-/divisors and \(v:A\setminus\{0\}\to \QQ^r\) is a discrete valuation on \(A\) for some \(r \in\ZZ_{>0}\).
The valuation \(v\) induces a filtration on \(A\) and one considers the associated graded algebra \(\gr_v(A)\), see below for details.
One now wants to find generators of \(A\) whose images in \(\gr_v(A)\) generate \(\gr_v(A)\) as a \(K\)\=/algebra.

In this article, we extend the work of Kaveh and Manon and consider a grading on \(A\) by an arbitrary finitely generated abelian group \(\Delta\).
We give a definition of a \emph{\(\Delta\)\=/homogeneous Khovanskii basis}, which -- as the name suggests -- should be a \(\Delta\)\=/homogeneous generating system of \(A\) that fulfils a `Khovanskii\-/like' property, see section~\ref{sec:homkhov} for the technical details.
We further bridge the gap between SAGBI bases and Khovanskii bases by allowing an ambient ring \(B\) of which \(A\) is a subalgebra.
In case of SAGBI bases, \(B\) would be a polynomial ring; for a Khovanskii basis as in \cite{KM19}, we have \(A = B\).

We further introduce a `Multiple Valuations Analogue to Khovanskii bases' (\emph{MUVAK bases}), where a Khovanskii\-/like property is required to be fulfilled for several valuations simultaneously.
Our construction is motivated by the computation of certain Cox rings.
In the past decade, there has been quite some interest in Cox rings of minimal models of quotient singularities \cite{Don16, DW17, Yam18, Gra19, Sch23} and their algorithmic construction contributes to the aim of having an `algorithmic minimal model programme', see \cite{Laz24}.
In \cite{Yam18}, Yamagishi gives an algorithm to compute generators of such a Cox ring and, in the language of this article, the algorithm in fact computes a MUVAK basis.

We describe our setting and give the definition of a homogeneous Khovanskii basis in section~\ref{sec:homkhov}.
We further present a `homogeneous version' of the \emph{subduction algorithm} (Algorithm~\ref{alg:homsubduc}), that is, an algorithm to reduce elements by generators of a subalgebra.
As in \cite{KM19}, this algorithm does in general not terminate.
However, we are able to prove its termination for every input under the condition that any \(\Delta\)\=/graded component of \(A\) is of finite dimension (Proposition~\ref{prop:deltafindim}), which notably includes the setting of a polynomial ring with the standard grading.
We further give an algorithm to compute a homogeneous Khovanskii basis, see Algorithm~\ref{alg:homkhov}.
The algorithm is similar to the one given in \cite{KM19} and relies in the same way on the equality of two ideals.
We prove the corresponding theorem (Theorem~\ref{thm:surjeqeq}) only under the assumption that the subduction algorithm (provably) terminates.

MUVAK bases are introduced in section~\ref{sec:muvak} and we show that a homogeneous MUVAK basis for just one valuation is a homogeneous Khovanskii basis (Theorem~\ref{thm:muvakkhov}).
We give an algorithm to compute such bases (Algorithm~\ref{alg:muvak}), however, only in the case where all valuations map to \(\ZZ\).

Just as for SAGBI and Khovanskii bases, a finite homogeneous Khovanskii basis or finite MUVAK basis does in general not exist for a given algebra.
In section~\ref{sec:ex}, we study several examples showing how the cardinality of a homogeneous Khovanskii basis may depend on the chosen grading.
In particular, we give an example for an algebra that has a finite homogeneous Khovanskii basis for one grading, but there is no finite basis for another one (Example~\ref{ex:inf}).

We finally discuss our initial motivation for this article, namely the computation of Cox rings of minimal models of quotient singularities in section~\ref{sec:cox}.
We recall a theorem by Yamagishi \cite{Yam18} and Grab \cite{Gra19} that characterizes a generating system for such a Cox ring by a MUVAK basis.
In \cite{Yam18}, Yamagishi gives an algorithm to compute such generators.
We implemented the algorithms in this paper in the computer algebra system OSCAR \cite{Osc,DEFHJ25} and compared these approaches with our implementation of Yamagishi's algorithm stemming from \cite{Sch23}.
Our algorithms appear to be an improvement of Yamagishi's algorithm, see the timings in Table~\ref{tab:timings}.
The code used for these timings is freely available online at \begin{center}\url{https://gitlab.com/math5724907/homogeneouskhovanskii}.\end{center}

In the appendix, we describe an efficient method to compute the homogenization of an ideal with respect to non\-/negative weights by a generalization of `Bayer's method'.
This is relevant for a performant implementation of both Algorithm~\ref{alg:muvak} and Yamagishi's algorithm.

\subsection*{Acknowledgements}
I thank Tommy Hofmann for helpful comments on an early version of this article.
This work was supported by the SFB\-/TRR~195 `Symbolic Tools in Mathematics and their Application' of the German Research Foundation (DFG).

\section{Homogeneous Khovanskii bases}
\label{sec:homkhov}

\subsection{Setting}

Throughout, let \(K\) be a field and let \(A\leq B\) be finitely generated \(K\)\=/algebras and domains.
Let \((\Gamma, +, \succeq)\) be a \emph{linearly ordered group}, that is, an abelian group \((\Gamma, +)\) equipped with a total ordering \(\succeq\) such that \(\gamma_1 \succeq \gamma_2\) implies \(\gamma_1 + \gamma_3 \succeq \gamma_2 + \gamma_3\) for all \(\gamma_1,\gamma_2,\gamma_3\in\Gamma\).

\begin{definition}[Valuations]
  A function \(v:B\setminus\{0\}\to \Gamma\) is a \emph{valuation over \(K\)} if \(v\) satisfies the following axioms:
  \begin{enumerate}[(i)]
    \item For all \(0\neq f, g\in B\) with \(0\neq f + g\), we have \(v(f + g) \succeq\min\{v(f), v(g)\}\).
    \item For all \(0\neq f, g\in B\), we have \(v(fg) = v(f) + v(g)\).
    \item For all \(0\neq f\in B\) and \(0\neq \lambda\in K\), we have \(v(\lambda f) = v(f)\).
  \end{enumerate}
\end{definition}

In the following, let \(v:B\setminus\{0\}\to \Gamma\) be a valuation over \(K\).
The valuation \(v\) gives a \(\Gamma\)\=/filtration \(\mathcal F_v = (F_{v\succeq \gamma})_{\gamma\in\Gamma}\) on \(B\) in the following way.
For \(\gamma\in\Gamma\), we set \[F_{v\succeq \gamma} \coloneqq \{f\in B\setminus \{0\}\mid v(f)\succeq\gamma\}\cup\{0\}.\]
Similarly, we define \(F_{v\succ \gamma}\) and obtain the corresponding associated graded algebra \[\gr_v(B) \coloneqq \bigoplus_{\gamma\in\Gamma}F_{v\succeq\gamma}/F_{v\succ\gamma}.\]
Finally, we assume that \(A\) is graded by a finitely generated abelian group \(\Delta\).
We denote the corresponding degree function by \(\deg_\Delta\).

\begin{definition}[Initial forms]
  \begin{enumerate}[(i)]
    \item Let \(f\in B\setminus\{0\}\) with valuation \(v(f) = \gamma\).
       We write \(\initial_v(f)\in\gr_v(B)\) for the residue class of \(f\) in \(F_{v\succeq\gamma}/F_{v\succ\gamma}\).
       We put \(\initial_v(0) = 0\in(\gr_v(B))_0\).
     \item Let \(f\in A\setminus\{0\}\) be \(\Delta\)\=/homogeneous of degree \(\deg_\Delta(f) = \delta \in \Delta\).
       We define \[\initial_v^\Delta(f) \coloneqq \initial_v(f)\otimes\delta \in \gr_v(B)\otimes_K K\Delta\] where \(K\Delta\) is the group ring of \(\Delta\).
  \end{enumerate}
\end{definition}

\begin{notation}
  For a subset \(S\subseteq A\), we denote by \(\initial_v^\Delta(S) \leq \gr_v(B)\otimes_K K\Delta\) the subalgebra generated by the set \(\{\initial_v^\Delta(f)\mid f\in S\text{ is \(\Delta\)\=/homogeneous}\}\).
\end{notation}

\begin{definition}[Homogeneous Khovanskii bases]
  Let \(\mathcal G\subseteq A\) be a set of \(\Delta\)\=/homogeneous generators of \(A\) as a \(K\)\=/algebra.
  We call \(\mathcal G\) a \emph{\(\Delta\)\=/homogeneous Khovanskii basis} of \(A\) with respect to \(v\), if the set \(\{\initial_v^\Delta(f)\mid f\in \mathcal G\}\) generates \(\initial_v^\Delta(A)\) as a \(K\)\=/algebra.
\end{definition}

We usually just write \emph{homogeneous Khovanskii basis} if the grading is clear from the context.

\begin{example}
  \label{ex:khovrun1}
  We use the following running example to illustrate the results in this section.
  To help with navigation between the examples, we note that this example is continued in \ref{ex:khovrun2} and \ref{ex:khovrun3}.
  Let \(K\) be a field of characteristic different from 2, \(B = K[x, y]\) and \(v:B\setminus\{0\}\to\ZZ\) the order of divisibility by \(x\) with the natural ordering on \(\ZZ\).
  We have \[\gr_v(B) = K[y] \oplus x K[y] \oplus x^2 K[y] \oplus \cdots = B.\]
  Let \(f_1 \coloneqq x + y,f_2 \coloneqq x - y\in B\) and set \(A = K[f_1,f_2]\).
  Clearly, \(A = B\).
  We endow \(A\) with a grading by \(\Delta \coloneqq \ZZ/2\ZZ\) via \(\deg_\Delta(f_1) = \bar 0\) and \(\deg_\Delta(f_2) = \bar 1\).
  We have \(v(f_1) = v(f_2) = 0\) and compute the initial forms:
  \[\begin{array}{c|c}
    & \initial_v^\Delta \\
    \hline
    f_1 & y\otimes\bar 0 \\
    f_2 & -y\otimes\bar 1
  \end{array}\]
  We claim that \(\{f_1,f_2\}\) is not a \(\Delta\)\=/homogeneous Khovanskii basis of \(A\) with respect to \(v\).
  Indeed, \(xy = \frac{1}{4}(f_1^2 - f_2^2)\in A\) is homogeneous of degree \(\deg_\Delta(xy) = \bar 0\) and hence \(\initial_v^\Delta(xy) = xy\otimes\bar 0\in\initial_v^\Delta(A)\).
  However, the initial forms \(\initial_v^\Delta(f_i)\) do not generate \(xy\otimes\bar 0\).
\end{example}

\begin{remark}
  Just as in \cite{KM19}, we do not require a homogeneous Khovanskii basis to be finite and in general a finite basis does not exist.
  In our application in section~\ref{sec:cox}, the existence of a finite homogeneous Khovanskii basis is guaranteed by deep theoretical arguments, see section~\ref{subsec:quotsing}.
\end{remark}

\begin{remark}
  Our definition of a homogeneous Khovanskii basis is a direct generalization of Khovanskii bases as defined in \cite{KM19}.
  More precisely, a Khovanskii basis in the sense of \cite{KM19} is a homogeneous Khovanskii basis in the special case where \(\Delta = \{0\}\) is the trivial group and \(A = B\).
\end{remark}

\begin{remark}
  \label{rem:sagbi}
  SAGBI bases \cite{RS90} are a special case of homogeneous Khovanskii bases as well.
  Here, \(B = K[X_1,\dots,X_n]\) is a polynomial ring, \(\Delta = \{0\}\) is the trivial group and the valuation \(v\) is induced by a monomial ordering.
  More precisely, a monomial ordering \(\geq\) on \(B\) induces a total linear ordering \(\succeq\) on \(\ZZ^n\) by setting \[\gamma_1 \succeq \gamma_2 \Longleftrightarrow X^{\gamma_1} \leq X^{\gamma_2}\] for \(\gamma_1,\gamma_2\in \ZZ^n\), where by \(X^{\gamma_i}\) we mean the monomial with exponents \(\gamma_i\).
  This gives a valuation \(v:B\setminus\{0\}\to \ZZ^n\) over \(K\) by putting \(v(f) = \gamma\) for \(f\in B\setminus\{0\}\) with leading monomial \(\operatorname{LM}_\geq(f) = X^\gamma\).
  Notice that by construction of \(\succeq\), the valuation of \(f\) is the \emph{smallest} exponent among the monomials of \(f\) with respect to \(\succeq\).
  In this way, the (ascending) filtration one associates to \(\geq\), see \cite[Example~6.5.4]{KR05}, is equivalent to the (descending) filtration we associate to \(v\).
  With this construction, we have \(\gr_v(B) = B\) and \(\initial_v(f) = \operatorname{LM}_\geq(f)\) for all \(f\in B\).
\end{remark}

Introducing the grading by \(\Delta\) is motivated by our application in section~\ref{sec:cox} and appears to be a natural extension of the existing theory.

\subsection{The homogeneous subduction algorithm}
An essential tool for the computation of SAGBI or Khovanskii bases is the so\-/called \emph{subduction algorithm} \cite[{}1.5]{RS90}, \cite[Algorithm~11.1]{Stu96}, \cite[Algorithm~2.11]{KM19}.
This algorithm is an analogue of division with remainder in a polynomial ring and allows one to write an element of the algebra as a polynomial in the given basis.
We now give a \(\Delta\)\=/homogeneous version of this algorithm.

We denote by \(\deg_v\) the degree function on \(\gr_v(B)\) corresponding to the grading by \(\Gamma\) and extend this grading trivially to the tensor product \(\gr_v(B)\otimes_K K\Delta\) by setting \(\deg_v(f\otimes \delta) \coloneqq \deg_v(f)\) for \(f\in \gr_v(B)\) homogeneous and \(\delta\in K\Delta\).
The group ring \(K\Delta\) is naturally graded by \(\Delta\) and we write \(\deg_\Delta\) for the corresponding degree function.
Again, we extend this grading to \(\gr_v(B)\otimes_K K\Delta\) in the natural way.

\begin{notation}
  For a polynomial \(h = \sum_{a\in\ZZ_{\geq 0}^n}\lambda_aX_1^{a_1}\cdots X_n^{a_n}\in K[X_1,\dots, X_n]\), we denote by \(\Supp(h) = \{\lambda_aX_1^{a_1}\cdots X_n^{a_n}\mid \lambda_a\neq 0\}\) the \emph{support} of \(h\), that is, the set of non\=/zero terms.
\end{notation}

We present the homogeneous subduction algorithm in Algorithm~\ref{alg:homsubduc}.
For full generality, we state the algorithm with a possibly infinite set \(\mathcal G\).
In the algorithm, we work with subrings of the polynomial ring \(K[X_f\mid f\in \mathcal G]\) with variables labelled by the elements of \(\mathcal G\), so \(X_{f_i}\) denotes a variable corresponding to \(f_i\in\mathcal G\).

\begin{algorithm}
  \caption{Homogeneous subduction}
  \label{alg:homsubduc}
  \Input{A \(\Delta\)\=/homogeneous element \(f\in A\) and a set \(\mathcal G\subseteq A\) of \(\Delta\)\=/homogeneous elements}
  \Output{A finite subset \(\{f_1,\dots,f_k\}\subseteq\mathcal G\), a polynomial \(h\in K[X_1,\dots, X_k]\) and an element \(r\in A\), which fulfil all of the following conditions:
  \begin{enumerate}[(1)]
    \item \label{alg:homsubduc:1}\(f = h(f_1,\dots, f_k) + r\),
    \item \label{alg:homsubduc:2}for every \(t\in\Supp(h)\), we have \(v(t(f_1,\dots,f_k))\succeq v(f)\) and \(\deg_\Delta(t(f_1,\dots,f_k)) = \deg_\Delta(f)\),
    \item \label{alg:homsubduc:3}either \(h = 0\) or \(v(r) \succ v(f)\),
    \item \label{alg:homsubduc:4}either \(r = 0\) or \(r\) is \(\Delta\)\=/homogeneous with \(\deg_\Delta(r) = \deg_\Delta(f)\),
    \item \label{alg:homsubduc:5}either \(r = 0\) or \(\initial_v^\Delta(r)\notin\initial_v^\Delta(\mathcal G)\).
  \end{enumerate}}
  \BlankLine
  \(\mathcal G'\coloneqq \emptyset\)\;
  \(h \coloneqq 0\)\;
  \(r \coloneqq f\)\;
  \While{\(r \neq 0\) and \(\initial_v^\Delta(r)\in\initial_v^\Delta(\mathcal G)\)}{
    Find a finite subset \(\{f_1,\dots,f_k\}\subseteq\mathcal G\) and a polynomial \(h'\in K[X_{f_1},\dots, X_{f_k}]\) with \(h'(\initial_v^\Delta(f_1),\dots,\initial_v^\Delta(f_k)) = \initial_v^\Delta(r)\), \(v(t(f_1,\dots,f_k)) \succeq v(r)\) and \(\deg_\Delta(t(f_1,\dots,f_k)) = \deg_\Delta(r)\) for every \(t\in\Supp(h')\)\;
    \(\mathcal G' \coloneqq \mathcal G'\cup \{f_1,\dots, f_k\}\)\;
    \(r \coloneqq r - h'(f_1,\dots, f_k)\)\;
    \(h \coloneqq h + h'\)\;
  }
  \Return{\(\mathcal G'\), \(h\), \(r\)}
\end{algorithm}

\begin{remark}
  \label{rem:alghomsubduc}
  Note that the conditions on the terms of \(h'\) in line 5 of Algorithm~\ref{alg:homsubduc} can be easily achieved.
  By assumption, \(\initial_v^\Delta(r)\in\initial_v^\Delta(\mathcal G)\), so there is a finite set \(\{f_1,\dots,f_k\}\subseteq\mathcal G\) and a polynomial \(h'\in K[X_1,\dots, X_k]\) with \[h'(\initial_v^\Delta(f_1),\dots,\initial_v^\Delta(f_k)) = \initial_v^\Delta(r).\]
  We endow the polynomial ring \(K[X_1,\dots, X_k]\) with a \(\Gamma\)- and a \(\Delta\)\=/grading by putting \(\deg_\Gamma(X_i) \coloneqq v(f_i)\) and \(\deg_\Delta(X_i) \coloneqq \deg_\Delta(f_i)\).
  We can now remove all terms \(t\in\Supp(h')\) from \(h'\) with \(\deg_\Gamma(t) \neq v(r)\) or \(\deg_\Delta(t) \neq \deg_\Delta(r)\).
  This does not change \(h'(\initial_v^\Delta(f_1),\dots,\initial_v^\Delta(f_k))\) by \(\Gamma\)- and \(\Delta\)\=/homogeneity of \(\initial_v^\Delta(r)\).
  In other words, \(h'\) is a homogeneous preimage of \(\initial_v^\Delta(r)\) under the \((\Gamma,\Delta)\)\=/graded morphism given by \(X_i\mapsto \initial_v^\Delta(f_i)\).
\end{remark}

\begin{proposition}
  Algorithm~\ref{alg:homsubduc} is correct assuming it terminates.
\end{proposition}
\begin{proof}
  Assume the algorithm terminates for a given input \(f\) and \(\mathcal G\) with a triple \(\mathcal G'\), \(h\) and \(r\).
  We need to show that \(\mathcal G'\), \(h\) and \(r\) fulfil the conditions \ref{alg:homsubduc:1} to \ref{alg:homsubduc:5} in Algorithm~\ref{alg:homsubduc}.

  If the algorithm terminates, condition \ref{alg:homsubduc:5} is fulfilled by construction of the \texttt{while}\=/loop.
  Further, condition \ref{alg:homsubduc:4} follows from \ref{alg:homsubduc:1} and \ref{alg:homsubduc:2}: indeed, \ref{alg:homsubduc:2} implies that \(h(f_1,\dots,f_k)\) is \(\Delta\)\=/homogeneous of degree \(\deg_\Delta(f)\) whenever \(h\neq 0\), hence so is \(r = f - h(f_1,\dots, f_k)\) by \ref{alg:homsubduc:1}.

  The conditions \ref{alg:homsubduc:1} to \ref{alg:homsubduc:3} are certainly fulfilled before the start of the \texttt{while}\=/loop.
  We prove that they are maintained during the run of the loop, that is, that after each iteration of the loop \ref{alg:homsubduc:1} to \ref{alg:homsubduc:3} hold.
  Let \(\mathcal G'\), \(h\) and \(r\) be the values at the beginning of a given iteration of the loop and let \(h'\) be the polynomial chosen during this iteration.
  Condition \ref{alg:homsubduc:1} is maintained by construction.
  By \ref{alg:homsubduc:3}, we have \(v(r)\succeq v(f)\) at the beginning of the loop, hence \(v(t(f_1,\dots,f_k)) \succeq v(f)\) for every \(t\in\Supp(h')\) and this gives \ref{alg:homsubduc:2} (see Remark~\ref{rem:alghomsubduc} for the feasibility of line 5 of the algorithm).

  Set \(s \coloneqq h'(f_1,\dots, f_k)\in A\).
  We have \(\initial_v(s) = h'(\initial_v(f_1), \dots, \initial_v(f_k)) = \initial_v(r)\), so \[\deg_v(\initial_v(r - s)) \succ \deg_v(\initial_v(r))\] and hence \(v(r - s) \succ v(r)\) which implies condition \ref{alg:homsubduc:3}.
\end{proof}

\begin{example}
  \label{ex:khovrun2}
  This example is part of the series also containing \ref{ex:khovrun1} and \ref{ex:khovrun3}.
  To illustrate Algorithm~\ref{alg:homsubduc}, we run it on a small example.
  Let \[f = \frac{1}{4}(f_1^2 - f_2^2) + f_1 = xy + x + y\in A\] and \(\mathcal G = \{f_1,f_2\}\).
  We see that \(\deg_\Delta(f) = \bar 0\in\ZZ/2\ZZ\).
  Then \(\initial_v^\Delta(f) = y\otimes\bar 0 = \initial_v^\Delta(f_1)\), so the first iteration of the subduction algorithm yields \(\mathcal G' = \{f_1\}\), \(h = X_1\) and \(r = f - f_1 = xy\).
  Now \(\initial_v^\Delta(f) = xy\otimes \bar 0\notin \initial_v^\Delta(\mathcal G)\) and the algorithm terminates with non\-/trivial remainder \(r\).
\end{example}

In general, there is no guarantee that Algorithm~\ref{alg:homsubduc} terminates, see Example~\ref{ex:nosub}.
We say that Algorithm~\ref{alg:homsubduc} \emph{terminates for} \(A\), if the algorithm terminates for every \(\Delta\)\=/homogeneous \(f\in A\) and every set \(\mathcal G\subseteq A\) of \(\Delta\)\=/homogeneous elements after finitely many steps.
We now give some conditions that ensure the termination.

\begin{remark}
  Assume that the set \(S(A, v) \coloneqq \{v(f)\mid 0\neq f\in A\}\) is \emph{maximum well\-/ordered}, that is, every subset of \(S(A, v)\) has a maximal element with respect to \(\succeq\).
  As in \cite[Proposition~2.13]{KM19}, we see that Algorithm~\ref{alg:homsubduc} terminates for \(A\) in this case as the valuation of \(r\) properly increases in every iteration.

  If \(B\) is a polynomial ring and \(v\) is induced by a monomial ordering \(>\) as in Remark~\ref{rem:sagbi}, then the maximum well\-/ordered property corresponds to \(>\) being a (minimum) well\-/ordering.
  In other words, \(>\) is a global ordering (in the terminology of \cite{GP08}) and termination is guaranteed in the setting of SAGBI bases.
\end{remark}

We need the following lemma for a different condition.
\begin{lemma}
  \label{lem:valbounded}
  Let \(0\neq V\leq A\) be a finite\-/dimensional \(K\)\=/vector subspace of \(A\).
  Then there is \(\gamma\in\Gamma\) such that \(v(a) \preceq \gamma\) for all \(a\in V\setminus\{0\}\).
\end{lemma}
\begin{proof}
  For \(\dim(V) = 1\), there is nothing to prove.
  Let \(\dim(V) = k\) and let \(a_1,\dots, a_k\in V\) be a basis.
  Write \(V' \coloneqq \langle a_1,\dots, a_{k - 1}\rangle_K\) for the subspace of dimension \(k - 1\).
  By induction, there exists \(\gamma\in\Gamma\) such that \(v(a) \preceq \gamma\) for all \(a\in V'\setminus\{0\}\).
  If also \(v(a) \preceq \gamma\) for all \(a\in V\setminus\{0\}\), we are done, so assume there are \(\lambda_1,\dots,\lambda_k\in K\), not all 0, such that for \(a \coloneqq \sum_{i = 1}^k\lambda_ia_i\) we have \(v(a) \succ \gamma\).
  In particular, \(a\notin V'\), so \(\lambda_k \neq 0\).

  We claim that \(v(a)\) gives an upper bound for the valuation on \(V\).
  Let \(0\neq b = \sum_{i = 1}^k\mu_ia_i\in V\) be any element.
  If \(\mu_k = 0\), then \(b\in V'\) and \(v(b) \preceq \gamma \prec v(a)\).
  Otherwise we compute
  \begin{align*}
    v(b) &= v\Big(a_k + \sum_{i = 1}^{k - 1}\frac{\mu_i}{\mu_k}a_i\Big) = v\bigg(\!a_k + \sum_{i = 1}^{k - 1}\frac{\lambda_i}{\lambda_k}a_i + \underbrace{\sum_{i = 1}^{k - 1}\Big(\frac{\mu_i}{\mu_k} - \frac{\lambda_i}{\lambda_k}\Big)a_i}_{\coloneqq c}\!\!\bigg)\\
         &= v(\lambda_k^{-1}a - c).
  \end{align*}
  If \(c = 0\), then \(v(b) = v(a)\).
  Otherwise \(v(c) \preceq \gamma\) since \(c\in V'\), so \(v(c) \neq v(a)\) and hence \(v(b) = \min\{v(a), v(c)\} = v(c)\).
\end{proof}

\begin{proposition}
  \label{prop:deltafindim}
  If the vector spaces \(A_\delta\) are of finite dimension for all \(\delta\in\Delta\), then Algorithm~\ref{alg:homsubduc} terminates for \(A\).
\end{proposition}
\begin{proof}
  During the algorithm, the element \(r\) remains in the graded component \(A_\delta\) with \(\delta = \deg_\Delta(f)\).
  The valuation of \(r\) properly increases in every iteration of the \texttt{while}-loop.
  By Lemma~\ref{lem:valbounded}, this must reach a maximum after finitely many steps.
\end{proof}

Note that Proposition~\ref{prop:deltafindim} applies in the case where \(A\) is a subalgebra of a polynomial ring that is standard graded or graded by positive weights.

\subsection{Computing homogeneous Khovanskii bases}
We generalize the algorithm for the computation of a Khovanskii basis from \cite{KM19} to our setting.
We prove correctness of the algorithm under a less restrictive assumption.

Let \(f_1,\dots,f_k\in A\) be \(\Delta\)\=/homogeneous generators of \(A\) as a \(K\)\=/algebra.
We consider the two morphisms \[\alpha:K[X_1,\dots, X_k]\to A,\ X_i\mapsto f_i\] and \[\beta:K[X_1,\dots,X_k]\to\initial_v^\Delta(A),\ X_i\mapsto\initial_v^\Delta(f_i).\]
We endow the polynomial ring \(K[X_1,\dots,X_k]\) with a grading by \(\Delta\) via \(\deg_\Delta(X_i) \coloneqq \deg_\Delta(f_i)\) and a grading by \(\Gamma\) via \(\deg_\Gamma(X_i) \coloneqq v(f_i) = \deg_v(\initial_v^\Delta(f_i))\).
The morphisms \(\alpha\) and \(\beta\) are \(\Delta\)\=/graded by construction and \(\beta\) is a \(\Gamma\)\=/graded morphism.

\begin{notation}
  We write \(\initial_\Gamma(h)\in K[X_1,\dots,X_k]\) for the sum of the terms of minimal \(\Gamma\)\=/degree of a polynomial \(h\in K[X_1,\dots, X_k]\) (with respect to \(\succeq\)).
\end{notation}

\begin{lemma}
  \label{lem:initial}
  Let \(h\in K[X_1,\dots, X_k]\).
  We have \(\deg_\Gamma(\initial_\Gamma(h)) \preceq v(\alpha(h))\).
  Further, we have
  \[\initial_\Gamma(h)(\initial_v(f_1),\dots,\initial_v(f_k)) =
    \begin{cases}
      \initial_v(\alpha(h))& \text{if }\deg_\Gamma(\initial_\Gamma(h)) = v(\alpha(h)),\\
      0 & \text{otherwise.}
  \end{cases}\]
\end{lemma}
\begin{proof}
  Write \(h = \sum_{a\in\ZZ_{\geq 0}^k}\lambda_a X_1^{a_1}\cdots X_k^{a_k}\).
  Then we have \[\deg_\Gamma(\initial_\Gamma(h)) = \min_{\substack{a\in\ZZ_{\geq 0}^k\\\lambda_a\neq 0}}v(f_1^{a_1}\cdots f_k^{a_k}) \preceq v(\alpha(h)).\]

  Let \(\gamma\coloneqq \deg_\Gamma(\initial_\Gamma(h))\).
  Assume \(\gamma \prec v(\alpha(h))\).
  Then \(\alpha(\initial_\Gamma(h)) \equiv 0\) in the quotient \(F_{v\succeq \gamma}/F_{v\succ\gamma}\), hence also \(\initial_\Gamma(h)(\initial_v(f_1),\dots,\initial_v(f_k)) = 0\).
  If \(\gamma = v(\alpha(h))\), then \[\initial_v(\alpha(h)) = \initial_v(\alpha(\initial_\Gamma(h))) = \initial_\Gamma(h)(\initial_v(f_1),\dots,\initial_v(f_k))\] as required.
\end{proof}

We can characterize a homogeneous Khovanskii basis via the equality of two ideals.
For this, let \(I\coloneqq \ker(\alpha)\) and set \(\initial_\Gamma(I) \coloneqq \langle \initial_\Gamma(h)\mid h\in I\rangle\ideal K[X_1,\dots, X_k]\).
Let further \(J\coloneqq \ker(\beta)\).

\begin{lemma}
  \label{lem:inIsubJw}
  We have \(\initial_\Gamma(I) \subseteq J\).
\end{lemma}
\begin{proof}
  Given a \(\Delta\)\=/homogeneous polynomial \(h\in I\), the initial form \(\initial_\Gamma(h)\) is \(\Delta\)\=/homogeneous as well.
  Hence \(\initial_\Gamma(I)\) is a \(\Delta\)\=/homogeneous ideal because \(I\) is \(\Delta\)\=/homogeneous.
  As also \(J\) is \(\Delta\)\=/homogeneous, it suffices to check the claim on \(\Delta\)\=/homogeneous elements.
  The remainder of the proof is now analogous to \cite[Lemma~2.16]{KM19}.
  Let \(h\in I\) be \(\Delta\)\=/homogeneous and set \(\gamma\coloneqq\deg_\Gamma(\initial_\Gamma(h))\).
  We have \(h(f_1,\dots,f_k) = 0\) and hence \[\initial_\Gamma(h)(\initial_v(f_1),\dots,\initial_v(f_k)) = 0\] by Lemma~\ref{lem:initial}.
  By \(\Delta\)\=/homogeneity, we conclude \(\beta(\initial_\Gamma(h)) = 0\), so \(\initial_\Gamma(h)\in J\).
\end{proof}

The following observation is an analogue of \cite[Lemma~6.2.11]{Sch23} which originates from \cite{Yam18}.
\begin{lemma}
  \label{lem:inI}
  Let \(0\neq h\in K[X_1,\dots,X_k]\) be \(\Gamma\)- and \(\Delta\)\=/homogeneous.
  We have \(h\in \initial_\Gamma(I)\) if and only if there exists a \(\Delta\)\=/homogeneous polynomial \(\tilde h\in K[X_1,\dots, X_k]\) such that \(h - \tilde h\in I\) and \(\deg_\Gamma(h) \prec \deg_\Gamma(\initial_\Gamma(\tilde h))\).
\end{lemma}
\begin{proof}
  Assume \(h\in \initial_\Gamma(I)\), so we can write \(h = \sum_{i = 1}^t\initial_\Gamma(h_i)\) with \(h_i\in I\).
  We may assume that the \(h_i\) are \(\Delta\)\=/homogeneous of same degree since \(h\) is \(\Delta\)\=/homogeneous and \(I\) is a \(\Delta\)\=/homogeneous ideal.
  Set \(h'\coloneqq\sum_{i = 1}^th_i\in I\).
  Then \(\initial_\Gamma(h') = h\) by \(\deg_\Gamma\)\=/homogeneity of \(h\).
  Hence \(\tilde h \coloneqq h - h'\) fulfils the requirements.

  Conversely, assume that we have a polynomial \(\tilde h\in K[X_1,\dots, X_k]\) as in the claim.
  Then it follows that \(\initial_\Gamma(h - \tilde h) = h\), so \(h\in \initial_\Gamma(I)\).
\end{proof}

\begin{lemma}
  \label{lem:inI2}
  Let \(h\in J\) be \(\Gamma\)- and \(\Delta\)\=/homogeneous.
  If Algorithm~\ref{alg:homsubduc} terminates with remainder 0 for \(h(f_1,\dots,f_k)\) and \(\{f_1,\dots,f_k\}\), then \(h\in \initial_\Gamma(I)\).
\end{lemma}
\begin{proof}
  Set \(f\coloneqq h(f_1,\dots,f_k)\).
  We have \(h(\initial_v(f_1),\dots,\initial_v(f_k)) = 0\), so \(\deg_\Gamma(h) \prec v(f)\).
  By assumption, Algorithm~\ref{alg:homsubduc} produces an element \(h'\in K[X_1,\dots,X_k]\) with \(f = h'(f_1,\dots,f_k)\).
  Further, by \ref{alg:homsubduc:2} in Algorithm~\ref{alg:homsubduc}, we have \(\deg_\Gamma(\initial_\Gamma(h')) \succeq v(f)\), from which we conclude \(\deg_\Gamma(\initial_\Gamma(h')) = v(f)\) by Lemma~\ref{lem:initial}.
  Hence \(h - h'\in I\) and \(\initial_\Gamma(h - h') = \initial_\Gamma(h) = h\), so \(h\in\initial_\Gamma(I)\).
\end{proof}

The following theorem is the equivalent of \cite[Theorem~2.17]{KM19}.
\begin{theorem}
  \label{thm:surjeqeq}
  Assume that Algorithm~\ref{alg:homsubduc} terminates for \(A\).
  Let \(f_1,\dots,f_k\in A\) be \(\Delta\)\=/homogeneous generators.
  The following conditions are equivalent:
  \begin{enumerate}[(i)]
    \item The morphism \(\beta\) is surjective, that is, \(\{f_1,\dots, f_k\}\) is a \(\Delta\)\=/homogeneous Khovanskii basis of \(A\).
    \item The ideals \(\initial_\Gamma(I)\) and \(J\) coincide.
    \item For a set of generators \(\{h_1,\dots,h_s\}\) of \(J\) consisting of \(\Gamma\)\=/homogeneous as well as \(\Delta\)\=/homogeneous polynomials, Algorithm~\ref{alg:homsubduc} terminates with remainder 0 for \(h_i(f_1,\dots,f_k)\) and \(\{f_1,\dots,f_k\}\), \(1\leq i\leq s\).
  \end{enumerate}
\end{theorem}
\begin{proof}
  `(i)\(\Longrightarrow\)(ii)': Let \(\beta\) be surjective.
  Then Algorithm~\ref{alg:homsubduc} terminates with remainder 0 for every \(f\in A\).
  So, \(\initial_\Gamma(I) = J\) by Lemma~\ref{lem:inIsubJw} and Lemma~\ref{lem:inI2}.

  `(ii)\(\Longrightarrow\)(i)': Assume \(\initial_\Gamma(I) = J\) and let \(0\neq f\in A\) be \(\Delta\)\=/homogeneous.
  We have to show that there is \(h\in K[X_1,\dots,X_k]\) with \(\beta(h) = \initial_v^\Delta(f)\).
  Let \(h'\in K[X_1,\dots,X_k]\) be \(\Delta\)\=/homogeneous with \(h'(f_1,\dots, f_k) = f\).
  If \(\deg_\Gamma(\initial_\Gamma(h')) = v(f)\), then \(\initial_\Gamma(h')(\initial_v(f_1),\dots,\initial_v(f_k)) = \initial_v(f)\) by Lemma~\ref{lem:initial}, so \(\initial_\Gamma(h')\) is the desired preimage of \(\initial_v^\Delta(f)\) under \(\beta\).
  Assume hence \(\deg_\Gamma(\initial_\Gamma(h')) \prec v(f)\), so \(\initial_\Gamma(h')(\initial_v(f_1),\dots,\initial_v(f_k)) = 0\) by Lemma~\ref{lem:initial} and \(\initial_\Gamma(h') \in J\).
  By assumption, \(\initial_\Gamma(h')\in \initial_\Gamma(I)\), so there is \(\tilde h\in K[X_1,\dots,X_k]\) with \(\initial_\Gamma(h') - \tilde h\in I\) and \(\deg_\Gamma(\initial_\Gamma(h')) \prec \deg_\Gamma(\initial_\Gamma(\tilde h))\) by Lemma~\ref{lem:inI}.
  Hence for \(h''\coloneqq h' - \initial_\Gamma(h') + \tilde h\) we have \(\alpha(h'') = f\) and \(\deg_\Gamma(\initial_\Gamma(h'')) \succ \deg_\Gamma(\initial_\Gamma(h'))\).
  Increasing the \(\Gamma\)\=/degree of the initial form of a preimage of \(f\) in this way, we eventually obtain a preimage \(h\in\alpha^{-1}(f)\) with \(\deg_\Gamma(\initial_\Gamma(h)) = v(f)\).
  Then \(\initial_\Gamma(h)\) gives the desired preimage under \(\beta\) as above.

  `(i)\(\Longrightarrow\)(iii)': By assumption, Algorithm~\ref{alg:homsubduc} terminates for \(A\).
  As \(\beta\) is surjective, the algorithm must return remainder 0 for every \(f\in A\).

  `(iii)\(\Longrightarrow\)(ii)': Let \(J = \langle h_1,\dots,h_s\rangle\) with \(\Delta\)\=/homogeneous and \(\Gamma\)\=/homogeneous polynomials \(h_i\) and assume Algorithm~\ref{alg:homsubduc} reduces \(h_i(f_1,\dots,f_k)\) to 0.
  This implies \(h_i\in\initial_\Gamma(I)\) by Lemma~\ref{lem:inI2}.
  Hence \(J\subseteq\initial_\Gamma(I)\) and the claim follows with Lemma~\ref{lem:inIsubJw}.
\end{proof}

We use Theorem~\ref{thm:surjeqeq} in Algorithm~\ref{alg:homkhov} to compute a homogeneous Khovanskii basis.

\begin{algorithm}
  \caption{Homogeneous Khovanskii basis}
  \label{alg:homkhov}
  \Input{\(\Delta\)\=/homogeneous generators \(f_1,\dots,f_k\in A\leq B\)}
  \Output{A \(\Delta\)\=/homogeneous Khovanskii basis of \(A\) with respect to \(v\)}
  \BlankLine
  Initialize \(\mathcal G \coloneqq \mathcal G' \coloneqq \{f_1,\dots, f_k\}\)\;
  \While{\(\mathcal G' \neq \emptyset\)}{
    \(\mathcal G' \coloneqq \emptyset\)\;
    With \(\mathcal G = \{f'_1,\dots, f'_l\}\), set up the morphism \[\beta:K[X_1,\dots,X_l]\to\initial_v^\Delta(A),\ X_i\mapsto \initial_v^\Delta(f'_i)\] and compute \(\Gamma\)\=/homogeneous as well as \(\Delta\)\=/homogeneous generators of the kernel \(\ker(\beta) = \langle h_1,\dots,h_s\rangle\)\;
    \For{\(h\in\{h_1,\dots,h_s\}\)}{
      \If{Algorithm~\ref{alg:homsubduc} terminates for \(h(f_1',\dots,f'_l)\) and \(\mathcal G\) with a non-zero remainder \(r\)}{
        \(\mathcal G' \coloneqq \mathcal G'\cup\{r\}\)\;
      }
    }
    \(\mathcal G \coloneqq \mathcal G\cup\mathcal G'\)\;
  }
  \Return \(\mathcal G\)\;
\end{algorithm}

\begin{proposition}
  \label{prop:homkhovcorrect}
  Assume that Algorithm~\ref{alg:homsubduc} terminates for \(A\).
  Algorithm~\ref{alg:homkhov} terminates after finitely many steps if and only if there exists a finite \(\Delta\)\=/homogeneous Khovanskii basis of \(A\) with respect to \(v\).
  If the algorithm terminates, it correctly returns such a basis.
\end{proposition}
\begin{proof}
  Assume that Algorithm~\ref{alg:homkhov} terminates with output \(\{f_1',\dots, f_l'\}\).
  By construction, these polynomials are generators of \(A\) as the set contains the original generators.
  Further, the \(f'_i\) are \(\Delta\)\=/homogeneous as Algorithm~\ref{alg:homsubduc} maintains the homogeneity.
  Let \(\beta:K[X_1,\dots,X_l]\to\initial_v^\Delta(A)\) be the map corresponding to the \(f_i'\) and set \(J \coloneqq\ker(\beta)\).
  As no generators were added in the final iteration of the \texttt{while}\=/loop, there are generators \(h_1,\dots,h_s\) of \(J\) such that Algorithm~\ref{alg:homsubduc} returns remainder 0 for \(h_i(f_1',\dots,f_l')\) and \(\{f_1',\dots,f_l'\}\).
  Hence \(\{f_1',\dots,f_l'\}\) is a \(\Delta\)\=/homogeneous Khovanskii basis of \(A\) by Theorem~\ref{thm:surjeqeq}.
  In particular, \(A\) admits such a finite basis.

  Conversely, assume that \(A\) contains a finite \(\Delta\)\=/homogeneous Khovanskii basis \(\{\tilde f_1,\dots,\tilde f_s\}\).
  Let \(f_1',\dots,f_l'\in A\) be the polynomials in a given iteration of the algorithm and let \(\beta\) be the corresponding map.
  We assume at first that for all \(i\in\{1,\dots,s\}\) we have \(\initial_v^\Delta(\tilde f_i)\in\im(\beta)\).
  Then \(\beta\) must be surjective as \(\initial_v^\Delta(\tilde f_i)\) generate \(\initial_v^\Delta(A)\) by construction.
  Hence \(\{f'_1,\dots,f'_l\}\) is a \(\Delta\)\=/homogeneous Khovanskii basis as well.
  By Theorem~\ref{thm:surjeqeq}, the generators of \(\ker(\beta)\) chosen by the algorithm must reduce to 0 in Algorithm~\ref{alg:homsubduc} and the algorithm terminates.

  Assume now that there is some \(f\in\{\tilde f_1,\dots,\tilde f_s\}\) with \(\initial_v^\Delta(f)\notin\im(\beta)\).
  As \(f\in A\), there is a \(\Delta\)\=/homogeneous polynomial \(h\in K[X_1,\dots,X_l]\) with \(h(f_1',\dots,f_l') = f\).
  By assumption, \(\initial_\Gamma(h)(\initial_v(f_1'),\dots,\initial_v(f_l')) \neq \initial_v(f)\), so \(\initial_\Gamma(h)(\initial_v(f_1'),\dots,\initial_v(f_l')) = 0\) by Lemma~\ref{lem:initial}.
  Hence \(\initial_\Gamma(h)\in J\).
  Let \(h_1,\dots,h_s\in K[X_1,\dots,X_k]\) with \(J = \langle h_1,\dots,h_s\rangle\) be the generators chosen by the algorithm.
  For every \(i\in\{1,\dots,s\}\), the algorithm computes \(h'_i\in K[X_1,\dots,X_l]\) and \(r_i\in A\) with \[h_i(f_1',\dots,f_l') = h'_i(f_1',\dots,f_l') + r_i\] using Algorithm~\ref{alg:homsubduc} with input \(h_i(f'_1,\dots,f'_l)\) and \(\{f'_1,\dots,f'_l\}\).
  Let \(f_{l + i}' \coloneqq r_i\).
  We have \(h_i^+ \coloneqq h_i' + X_{l + i}\in K[X_1,\dots,X_{l + s}]\) with \(h_i^+(f_1',\dots,f_{l + s}') = h_i(f_1',\dots,f_l')\) in the next iteration of the algorithm.
  Then \(h_i - h_i^+\in I\), where \(I\) is the kernel of \[K[X_1,\dots,X_{l + s}]\to A,\ X_i\mapsto f_i'.\]
  Hence \(\initial_\Gamma(h_i - h_i^+) = \initial_\Gamma(h_i) = h_i\) by \ref{alg:homsubduc:2} in Algorithm~\ref{alg:homsubduc} and the \(\Gamma\)\=/homogeneity of \(h_i\).
  This implies \(h_i\in \initial_\Gamma(I)\).
  Recall that \(\initial_\Gamma(h)\in\langle h_1,\dots, h_s\rangle\), so we conclude \(\initial_\Gamma(h)\in \initial_\Gamma(I)\).
  By Lemma~\ref{lem:inI}, there is \(\tilde h\in K[X_1,\dots,X_{l + s}]\) with \(\initial_\Gamma(h) - \tilde h\in I\) and \(\deg_\Gamma(\initial_\Gamma(h)) \prec \deg_\Gamma(\initial_\Gamma(\tilde h))\).
  Then for \(h^+ \coloneqq h - \initial_\Gamma(h) + \tilde h\) we have \(h^+(f_1',\dots,f_l') = f\) and \(\deg_\Gamma(\initial_\Gamma(h)) \prec \deg_\Gamma(\initial_\Gamma(h^+))\).
  Increasing the \(\Gamma\)\=/degree of a preimage of \(f\) in this way, we find a preimage \(h\) of degree \(\deg_\Gamma(\initial_\Gamma(h)) = v(f)\) after finitely many iterations.
  Equivalently, we achieve \(\initial_\Gamma^\Delta(f)\in\im(\beta)\) (with an extended map \(\beta\)) after finitely many iterations.
\end{proof}

\begin{example}
  \label{ex:khovrun3}
  This example is part of the series also containing \ref{ex:khovrun1} and \ref{ex:khovrun2}.
  Let \(\alpha\) and \(\beta\) be the maps corresponding to \(\{f_1,f_2\}\), that is, \(\alpha(X_i) = f_i\) and \(\beta(X_i) = \initial_v^\Delta(f_i)\).
  First of all, we have \(I = \ker(\alpha) = \{0\}\), so \(\initial_\Gamma(I) = \{0\}\), but \(J = \ker(\beta) = \langle X_1^2 - X_2^2\rangle\).
  By Theorem~\ref{thm:surjeqeq}, this confirms our observation that \(\{f_1,f_2\}\) is not a \(\Delta\)\=/homogeneous Khovanskii basis.

  We run Algorithm~\ref{alg:homkhov} with \(\mathcal G = \{f_1,f_2\}\).
  We have \(J = \langle X_1^2 - X_2^2\rangle\), so we must run Algorithm~\ref{alg:homsubduc} with \(f_3 \coloneqq f_1^2 - f_2^2 = 4xy\) and \(\{f_1,f_2\}\).
  Because \(\initial_v^\Delta(f_3) = 4xy\otimes\bar 0\), the algorithm immediately terminates with remainder \(f_3\).

  So we set \(\mathcal G = \{f_1,f_2,f_3\}\).
  Computing the relations of the initial forms \(\initial_v^\Delta(f_i)\) yields the same ideal \(J\) as before, so Algorithm~\ref{alg:homkhov} terminates with the set of generators \(\mathcal G\).

  We confirm the result using Theorem~\ref{thm:surjeqeq}.
  For the relations of the polynomials \(f_i\), we obtain the ideal \(I = \langle X_1^2 - X_2^2 - X_3\rangle\).
  The grading by \(\Gamma = \ZZ\) on \(K[X_1,X_2,X_3]\) is given by \(\deg_\Gamma(X_i) = v(f_i)\), so \(\deg_\Gamma(X_1) = 0\), \(\deg_\Gamma(X_2) = 0\) and \(\deg_\Gamma(X_3) = 1\).
  Hence we have \(\initial_\Gamma(I) = \langle X_1^2 - X_2^2\rangle\) and this coincides with \(J\) as required.
\end{example}

\section{A generalization to multiple valuations}
\label{sec:muvak}

\subsection{MUVAK bases}
We now consider multiple valuations \(v_1,\dots,v_m:B\setminus\{0\}\to \Gamma\) over \(K\) simultaneously.

\begin{definition}[Faithful representability]
  Let \(f_1,\dots, f_k\in A\) be \(\Delta\)\=/homogeneous generators and let \(f\in A\) be a \(\Delta\)\=/homogeneous element.
  We say that \(f\) is \emph{faithfully representable} in \(f_1,\dots, f_k\) with respect to \(v_1,\dots,v_m\), if there is a polynomial \(h\in K[X_1,\dots,X_k]\) such that \(h(f_1,\dots,f_k) = f\) and for all \(t\in\Supp(h)\) and for all \(i\in \{1,\dots,m\}\), we have \(v_i(t(f_1,\dots,f_k)) \succeq v_i(f)\).
\end{definition}

\begin{definition}[Homogeneous MUVAK bases]
  Let \(\mathcal G \subseteq A\) be a set of \(\Delta\)\=/homogeneous generators of \(A\) as a \(K\)\=/algebra.
  We say that \(\mathcal G\) is a \emph{\(\Delta\)\=/homogeneous MUVAK basis}\footnote{\textbf{M{\kern -.5pt}u}ltiple \textbf{V}{\kern-1pt}aluations \textbf{A}nalogue to \textbf{K}hovanskii bases} of \(A\) with respect to \(v_1,\dots,v_m\), if for every \(\Delta\)\=/homogeneous \(f\in A\) there are \(f_1,\dots, f_k\in\mathcal G\), such that \(f\) is faithfully representable in \(f_1,\dots,f_k\) with respect to \(v_1,\dots,v_m\).
\end{definition}

While we take full responsibility for inventing the name, \(\Delta\)\=/homogeneous MUVAK bases already appeared in the literature in the construction of certain Cox rings.
We give more details on this application in Section~\ref{sec:cox}.

\begin{example}
  \label{ex:muvakrun1}
  We again illustrate the results of this section with a running example, which is continued in \ref{ex:muvakrun2}, \ref{ex:muvakrun3} and \ref{ex:muvakrun4}.
  Let \(K\) be a field, \(B = K[x, y]\) and \(v_x\) and \(v_y\) be the order of divisibility by \(x\) and \(y\), respectively, with the natural ordering on \(\ZZ\).
  Let \[f_1 \coloneqq x^2 + xy,f_2 \coloneqq x^2, f_3 \coloneqq y^2 + xy, f_4 \coloneqq y^2\in B\] and set \(A = K[f_1,f_2, f_3, f_4]\), so \(A\) is the subalgebra generated by all monomials of even degree.
  For simplicity, we do not consider a grading in this example and put \(\Delta = \{0\}\) (although everything works with the standard grading by \(\Delta = \ZZ\)).
  We have the following valuations and initial forms:
  \[\def\arraystretch{1.2}\begin{array}{c|c|c}
    & v_x & v_y \\
    \hline
    f_1 & 1 & 0 \\
    f_2 & 2 & 0 \\
    f_3 & 0 & 1 \\
    f_4 & 0 & 2
  \end{array}\quad\quad\quad
  \begin{array}{c|c|c}
    & \initial_{v_x} & \initial_{v_y} \\
    \hline
    f_1 & xy & x^2 \\
    f_2 & x^2 & x^2 \\
    f_3 & y^2 & xy \\
    f_4 & y^2 & y^2
  \end{array}\]
  We claim that \(\{f_1,\dots,f_4\}\) is not a MUVAK basis of \(A\) with respect to \(v_x,v_y\).
  Consider \(f = f_1 - f_2 = xy\in A\).
  For degree reasons, any polynomial \(h\in K[X_1,\dots, X_4]\) with \(h(f_1,f_2,f_3,f_4) = xy\) must be of the form \(h = a_1X_1 + a_2X_2 + a_3X_3 + a_4X_4\) with \(a_i\in K\).
  Furthermore, we must have \(a_1 \neq 0\) or \(a_3\neq 0\) because only the polynomials \(f_1\) and \(f_3\) involve the monomial \(xy\).
  If \(a_1X_1\in \Supp(h)\), then \(v_y(a_1f_1) < v_y(xy) = 1\), and if \(a_3X_3\in\Supp(h)\), then \(v_x(a_3f_3) < v_x(xy) = 1\).
  Hence \(f\) is not faithfully representable in \(f_1,\dots, f_4\) with respect to \(v_x\) and \(v_y\).
\end{example}

Before we present an algorithm for the computation of MUVAK bases, we relate MUVAK bases to Khovanskii bases, see Theorem~\ref{thm:muvakkhov} for the precise statement.
For this, as in the previous section, we let \(f_1,\dots,f_k\in A\) be \(\Delta\)\=/homogeneous generators of \(A\) as a \(K\)\=/algebra and consider the morphism \[\alpha:K[X_1,\dots, X_k]\to A,\ X_i\mapsto f_i.\]
We endow the polynomial ring \(K[X_1,\dots,X_k]\) with a grading by \(\Delta\) via \(\deg_\Delta(X_i) \coloneqq \deg_\Delta(f_i)\).
We also have \(m\) morphisms \[\beta_i:K[X_1,\dots,X_k]\to\initial_{v_i}^\Delta(A),\ X_j\mapsto\initial_{v_i}^\Delta(f_j)\] and this gives \(\Gamma\)\=/gradings via \(\deg_i(X_j) \coloneqq v_i(f_j)\) for all \(i\in \{1,\dots,m\}\).
The morphisms \(\alpha\) and \(\beta_i\) are \(\Delta\)\=/graded by construction and \(\beta_i\) is a \(\Gamma\)\=/graded morphism via \(\deg_i\).

\begin{notation}
  We write \(\initial_i(h)\in K[X_1,\dots,X_k]\) for the sum of the terms of minimal \(\deg_i\)\=/degree of a polynomial \(h\in K[X_1,\dots, X_k]\) (with respect to \(\succeq\)).
\end{notation}

We have the following equivalent characterization of a MUVAK basis using the maps \(\alpha\) and \(\beta_i\).
\begin{lemma}
  \label{lem:muvakbeta}
  The set \(\{f_1,\dots,f_k\}\) is a \(\Delta\)\=/homogeneous MUVAK basis of \(A\) with respect to \(v_1,\dots,v_m\) if and only if for every \(\Delta\)\=/homogeneous \(f\in A\) there is a \(\Delta\)\=/homogeneous \(h\in K[X_1,\dots,X_k]\) with \(\alpha(h) = f\) and \(\beta_i(\initial_i(h)) = \initial_{v_i}^\Delta(f)\) for all \(i\in\{1,\dots,m\}\).
\end{lemma}
\begin{proof}
  Let \(f\in A\) be \(\Delta\)\=/homogeneous.
  If \(\{f_1,\dots,f_k\}\) is a \(\Delta\)\=/homogeneous MUVAK basis with respect to \(v_1,\dots,v_m\), then \(f\) must be faithfully representable in \(f_1,\dots,f_k\), so there is \(h\in K[X_1,\dots, X_k]\) with \(\alpha(h) = f\) and \(\deg_i(\initial_i(h)) = v_i(f)\) for all \(i \in \{1,\dots,m\}\).
  The maps \(\beta_i\) are \(\deg_i\)\=/graded, so we must have \(\beta_i(\initial_i(h)) = \initial_{v_i}^\Delta(f)\) for degree reasons.

  Conversely, assume that there is a \(\Delta\)\=/homogeneous \(h\in K[X_1,\dots, X_k]\) with \(\alpha(h) = f\) and \(\beta_i(\initial_i(h)) = \initial_{v_i}^\Delta(f)\) for all \(i\in\{1,\dots,m\}\).
  Then \(\deg_i(\initial_i(h)) = \deg_{v_i}(\initial_{v_i}(f)) = v_i(f)\) as required.
\end{proof}

\begin{theorem}
  \label{thm:muvakkhov}
  Let \(v:B\setminus\{0\}\to \Gamma\) be a valuation over \(K\) and let \(f_1,\dots,f_k\in A\) be \(\Delta\)\=/homogeneous generators.
  If \(\{f_1,\dots,f_k\}\) is a \(\Delta\)\=/homogeneous MUVAK basis of \(A\) with respect to \(v\), then \(\{f_1,\dots,f_k\}\) is a \(\Delta\)\=/homogeneous Khovanskii basis of \(A\) with respect to \(v\).

  Conversely, if \(\{f_1,\dots,f_k\}\) is a \(\Delta\)\=/homogeneous Khovanskii basis of \(A\) with respect to \(v\) and Algorithm~\ref{alg:homsubduc} terminates for \(A\), then \(\{f_1,\dots,f_k\}\) is a \(\Delta\)\=/homogeneous MUVAK basis of \(A\) with respect to \(v\).
\end{theorem}
\begin{proof}
  We write \[\beta:K[X_1,\dots,X_k]\to \initial_v^\Delta(A),~X_i\mapsto\initial_v^\Delta(f_i)\] and let \(\deg_\Gamma\) and \(\initial_\Gamma\) be as in Section~\ref{sec:homkhov}.

  Assume that \(\{f_1,\dots,f_k\}\) is a \(\Delta\)\=/homogeneous MUVAK basis and let \(f \in A\) be \(\Delta\)\=/homogeneous.
  By Lemma~\ref{lem:muvakbeta}, there is a \(\Delta\)\=/homogeneous polynomial \(h\in K[X_1,\dots,X_k]\) with \(\alpha(h) = f\) and \(\beta(\initial_\Gamma(h)) = \initial_v^\Delta(f)\).
  Hence \(\beta\) is surjective and \(\{f_1,\dots,f_k\}\) is a \(\Delta\)\=/homogeneous Khovanskii basis.

  Assume now that \(\{f_1,\dots,f_k\}\) is a \(\Delta\)\=/homogeneous Khovanskii basis or, equivalently, that \(\beta\) is surjective.
  Let \(f\in A\) be \(\Delta\)\=/homogeneous.
  We run Algorithm~\ref{alg:homsubduc} for \(f\) together with \(\{f_1,\dots,f_k\}\).
  By assumption, the algorithm terminates, and by surjectivity of \(\beta\), the algorithm must give remainder \(r = 0\) and an element \(h\in K[X_1,\dots,X_k]\) with \(f = h(f_1,\dots,f_k)\) and \(\deg_\Gamma(\initial_\Gamma(h)) = v(f)\).
  Hence \(f\) is faithfully representable in \(f_1,\dots,f_k\) with respect to \(v\) via \(h\).
\end{proof}

\begin{remark}
  It appears that being a MUVAK basis with respect to a single valuation is a stronger property than being a Khovanskii basis.
  It is not clear what an analogue for the filtration given by one valuation in the case of multiple valuations should be.
  Finding such an analogue appears to be crucial for fully integrating MUVAK bases into the theory of SAGBI and Khovanskii bases.
\end{remark}

\begin{remark}
  If \(\{f_1,\dots,f_k\}\) is a \(\Delta\)\=/homogeneous MUVAK basis with respect to \(v_1,\dots,v_m\), then \(\{f_1,\dots,f_k\}\) is a \(\Delta\)\=/homogeneous MUVAK basis with respect to every single valuation \(v_i\) as well.
  Hence \(\{f_1,\dots,f_k\}\) is a \(\Delta\)\=/homogeneous Khovanskii basis with respect to every valuation \(v_i\) by Theorem~\ref{thm:muvakkhov}.

  The converse statement is in general wrong: The set \(\{f_1,\dots,f_4\}\) in Example~\ref{ex:muvakrun1} is a Khovanskii basis for \(A\) with respect to \(v_x\) as well as with respect to \(v_y\), but it is not a MUVAK basis with respect to \(v_x\) and \(v_y\).
\end{remark}

\subsection{Negative homogenization}

To give an algorithm for the computation of a \(\Delta\)\=/homogeneous MUVAK basis, we require the notion of `negative homogenization'.

Let \(R\coloneqq K[X_1,\dots, X_k]\) be a polynomial ring graded by a matrix \(W\in\ZZ^{m\times k}\) of rank \(m \geq 1\), that is, \(\deg_W(X_j) \coloneqq (W_{1j},\dots, W_{mj})\in\ZZ^m\), \(1\leq j\leq k\).
We introduce additional variables \(t_1,\dots, t_m\) and endow the ring \(R^- \coloneqq K[X_1,\dots, X_k, t_1,\dots, t_m]\) with the grading induced by the matrix \(W^- \coloneqq (W \mid -I_m)\), where \(I_m\) denotes the identity matrix of rank \(m\).
Analogous to the homogenization with respect to variables of degree 1, see for example \cite[Definition~4.3.1]{KR05}, we define the homogenization of a polynomial \(f\in R\) with respect to the variables \(t_1,\dots, t_m\) of degree \(-1\).
\begin{definition}[Minimal degree and negative homogenization]
  Let \(f\in R\setminus \{0\}\) and write \(f = f_1 + \cdots + f_s\) with \(f_i\) the terms of \(f\).
  For \(j = 1,\dots, s\), let \(\deg_W(f_j) = (d_{1j},\dots, d_{mj})\in\ZZ^m\).
  Moreover, for \(i = 1,\dots, m\), let \(\mu_i \coloneqq \min\{d_{ij}\mid j = 1,\dots, s\}\).
  \begin{enumerate}
    \item The tuple \((\mu_1,\dots, \mu_m)\) is called the \emph{minimal degree} of \(f\) with respect to the grading given by \(W\) and denoted by \(\mindeg_W(f)\).
    \item The \emph{homogenization} of \(f\) with respect to the grading given by \(W\) is the polynomial \[f^\hom = \sum_{j = 1}^s f_j t_1^{d_{1j} - \mu_1}\cdots t_m^{d_{mj} - \mu_m}\in R^-.\]
      For the zero polynomial, we set \(0^\hom = 0\).
  \end{enumerate}
\end{definition}
The polynomial \(f^\hom\) is homogeneous of degree \(\mindeg_W(f)\) by construction.

We connect the negative homogenization to the homogenization in the usual sense, that is, with respect to variables of positive degree.
\begin{lemma}
  \label{lem:neghom}
  Let \(R^+ \coloneqq K[X_1,\dots, X_k, u_1,\dots, u_r]\) be a polynomial ring graded by the matrix \(W^+ \coloneqq (- W \mid I_r)\).
  For a polynomial \(f\in R\setminus \{0\}\), let \(f_-\), respectively \(f_+\), be the homogenization of \(f\) as an element of \(R^-\), respectively \(R^+\), with respect to \(t_1,\dots, t_r\), respectively \(u_1,\dots, u_r\).
  Then \(f_- = f_+(X_1,\dots, X_k, t_1,\dots, t_r)\).
\end{lemma}
\begin{proof}
  Write \(f = f_1 + \cdots + f_s\) with \(f_i\) the terms of \(f\).
  Let \(\mu_+\coloneqq\operatorname{topdeg}_{-W}(f)\) (see \cite[Definition~4.3.1]{KR05}) and \(\mu_-\coloneqq\mindeg_W(f)\).
  We observe that \(-\mu_+ = \mu_-\).
  We have
  \begin{align*}
    f_- &= \sum_{j = 1}^s\Big(f_j\prod_{i = 1}^rt_i^{\deg_W(f_j)_i - (\mu_-)_i}\Big) = \sum_{j = 1}^s\Big(f_j\prod_{i = 1}^rt_i^{-\deg_{-W}(f_j)_i + (\mu_+)_i}\Big)\\
    &= f_+(X_1,\dots, X_k, t_1,\dots, t_r),
  \end{align*}
  as required.
\end{proof}

\begin{remark}
  \label{rem:neghom}
  Lemma~\ref{lem:neghom} enables us to use well\-/known `calculation rules' for homogenized polynomials and ideals as in \cite[Proposition~4.3.2]{KR05} and \cite[Proposition~4.3.5]{KR05}, although these rules are only proven for homogenizations with respect to variables of positive degree in the given references.
\end{remark}

Given a (homogeneous) polynomial \(f\in K[X_1,\dots,X_k, t_1,\dots,t_m]\), we write \(f^\deh \coloneqq f|_{t_j = 1}\in K[X_1,\dots,X_k]\) for the dehomogenized polynomial.
For an ideal \(I\ideal K[X_1,\dots,X_k]\), we write \(I^\hom \coloneqq \langle f^\hom\mid f\in I\rangle\ideal K[X_1,\dots,X_k, t_1,\dots,t_m]\), as usual.

\subsection{Computing homogeneous MUVAK bases}
Yamagishi \cite{Yam18} gives an algorithm for the computation of a \(\Delta\)\=/homogeneous MUVAK basis in the special situation discussed in section~\ref{sec:cox}.
In the following, we present an algorithm that appears to perform better than Yamagishi's algorithm.
However, just as in \cite{Yam18}, the algorithm is restricted to the case \(\Gamma = \ZZ\).
From now on, we hence assume that \(\Gamma = \ZZ\) with the natural ordering \(\geq\).

Let \(f_1,\dots,f_k\in A\) be \(\Delta\)\=/homogeneous generators of \(A\) as a \(K\)\=/algebra and consider the morphisms \(\alpha\) and \(\beta_i\) for \(i\in\{1,\dots,m\}\) as above.
Let \(R\coloneqq K[X_1,\dots,X_k]\).
We collect the gradings \(\deg_i\) on \(R\) in a matrix \(W\in\ZZ^{m\times k}\) via \(W_{ij} \coloneqq \deg_i(X_j)\).
In the following, the homogenization operator \((\cdot)^\hom\) always means the negative homogenization with respect to the grading given by \(W\) (mapping from \(R\) to \(R[t_1,\dots,t_m]\)).
We extend the gradings \(\deg_i\) to \(R[t_1,\dots,t_m]\) by putting \(\deg_i(t_j) = -\delta_{ij}\) for \(1\leq i, j\leq m\), where \(\delta_{ij}\) denotes the Kronecker delta.

\begin{notation}
  For \(h\in K[X_1,\dots,X_k]\) and \(i\in\{1,\dots,m\}\), we write \[\mindeg_i(h) = \min_{t\in\Supp(h)}\deg_i(t).\]
\end{notation}

We observe:
\begin{lemma}
  \label{lem:degivi}
  Let \(f\in A\) be \(\Delta\)\=/homogeneous and let \(h\in\alpha^{-1}(f)\).
  Then \[\deg_i(h^\hom) = \mindeg_i(h) \leq v_i(f)\] for all \(i\in\{1,\dots, m\}\).
\end{lemma}
\begin{proof}
  The first equality holds because \((\cdot)^\hom\) denotes a negative homogenization.
  The inequality follows from the definition of the gradings \(\deg_i\).
\end{proof}

Let \(I \coloneqq \ker(\alpha)\).
We have the following characterization of faithful representability.
\begin{lemma}
  \label{lem:faithrep}
  Let \(f\in A\) be \(\Delta\)\=/homogeneous, let \(h\in \alpha^{-1}(f)\) be any preimage and put \[\theta\coloneqq\prod_{i = 1}^m t_i^{v_i(f) - \deg_i(h^\hom)}.\]
  Then \(f\) is faithfully representable in \(f_1,\dots,f_k\) with respect to \(v_1,\dots,v_m\) if and only if \(h^\hom\in I^\hom + \langle\theta\rangle\).
\end{lemma}
\begin{proof}
  Let \(f\) be faithfully representable and let \(h'\in \alpha^{-1}(f)\) such that for all \(t\in\Supp(h')\) and for all \(i\in \{1,\dots,m\}\), we have \(v_i(t(f_1,\dots,f_k)) \geq v_i(f)\).
  As \(\alpha(h) = \alpha(h')\), we have \(h - h'\in I\), so \((h - h')^\hom\in I^\hom\).
  By assumption, we have \(\deg_i(h'^\hom) = v_i(f)\) and by Lemma~\ref{lem:degivi}, we have \(\mindeg_i(h) \leq \mindeg_i(h')\) for all \(i\in \{1,\dots, m\}\).
  Hence \((h - h')^\hom = h^\hom - h'^\hom\theta\) and therefore \[h^\hom = (h - h')^\hom + h'^\hom\theta \in I^\hom + \langle \theta\rangle.\]

  Conversely, let \(h^\hom\in I^\hom + \langle \theta\rangle\).
  Then there are \(h'\in R[t_1,\dots,t_m]\) and \(h''\in I^\hom\) with \(h^\hom = h'' + h'\theta\).
  We may assume that \(h'\) is \(\deg_i\)\=/homogeneous for all \(i\in\{1,\dots,m\}\) with \(\deg_i(h') = \deg_i(h^\hom) - \deg_i(\theta) = v_i(f)\).
  We hence have \(\deg_i(t) \geq v_i(f)\) for all \(t\in \Supp(h'^\deh)\).
  Further \(h - h'^\deh\in I\) by \cite[Proposition~4.3.5~b)]{KR05}, so \(f\) is faithfully representable via \(h'^\deh\).
\end{proof}

\begin{example}
  \label{ex:muvakrun2}
  This example is part of the series also containing \ref{ex:muvakrun1}, \ref{ex:muvakrun3} and \ref{ex:muvakrun4}.
  Let again \(f = xy\).
  The homogenization matrix is given by \[W = \begin{pmatrix} 1 & 2 & 0 & 0\\ 0 & 0 & 1 & 2 \end{pmatrix}.\]
  Computing the relations of \(f_1,\dots,f_4\) gives \[I = \langle X_1 - X_2 - X_3 + X_4, X_3^2 - X_2X_4 - 2X_3X_4 + X_4^2\rangle\] with homogenization \[I^\hom = \langle X_1t_1 - X_2t_1^2 - X_3t_2 + X_4t_2^2, X_3^2 - X_2X_4t_1^2 - 2X_3X_4t_2 + X_4^2t_2^2\rangle.\]
  For the polynomial \(h = X_1 - X_2\) with \(h(f_1,\dots,f_4) = f\), we have \(h^\hom = X_1 - X_2t_1\), \(\deg_1(h^\hom) = 1\) and \(\deg_2(h^\hom) = 0\).
  Hence \(\theta = t_2\) and we see that \(h^\hom\notin I^\hom + \langle t_2\rangle\).
  So, Lemma~\ref{lem:degivi} confirms that \(f\) is not faithfully representable in \(f_1,\dots,f_4\).
\end{example}

\begin{lemma}
  \label{lem:faithrep2}
  Let \(f\in A\) be \(\Delta\)\=/homogeneous.
  If for all \(h\in \alpha^{-1}(f)\) and for all \(i\in \{1,\dots,m\}\) we have \(\mindeg_i(h) = v_i(f)\) or \(\initial_i(h)^\hom \in I^\hom + \langle t_i\rangle\), then \(f\) is faithfully representable in \(f_1,\dots,f_k\) with respect to \(v_1,\dots,v_m\).
\end{lemma}
\begin{proof}
  Let \(h\in \alpha^{-1}(f)\).
  If \(\mindeg_i(h) = v_i(f)\) for all \(i \in\{1,\dots, m\}\), then \(f\) is faithfully representable in \(f_1,\dots,f_k\) via \(h\).
  So, assume there is \(i\in\{1,\dots,m\}\) with \(\mindeg_i(h) < v_i(f)\), hence \(\initial_i(h)^\hom \in I^\hom + \langle t_i\rangle\) by assumption.
  Then there is \(h'\in \langle t_i\rangle\) with \(\initial_i(h)^\hom - h'\in I^\hom\), so \(\initial_i(h) - h'^\deh \in I\).
  We may assume \(\deg_j(h') = \deg_j(\initial_i(h)^\hom)\) for all \(j\in\{1,\dots,m\}\), so \(\mindeg_j(h'^\deh) \geq \mindeg_j(\initial_i(h))\) and \(\mindeg_i(h'^\deh) > \deg_i(\initial_i(h))\) as \(t_i\mid h'\).
  We now replace \(h\) by \(h - \initial_i(h) + h'^\deh\).
  By continuing to increase the \(\deg_i\)\=/degree in the described way, we arrive at a polynomial \(\tilde h\in\alpha^{-1}(f)\) with \(\mindeg_i(\tilde h) = v_i(f)\) and \(\mindeg_j(\tilde h) \geq \mindeg_j(h)\) for all \(j\neq i\) after finitely many steps.
  We may now continue this process with another index \(i'\) for which \(\mindeg_{i'}(\tilde h) < v_{i'}(f)\) and eventually we obtain a polynomial \(h'\in R\) with \(\alpha(h') = f\) and \(\mindeg_j(h') = v_j(f)\) for all \(j\in \{1,\dots, m\}\), so \(f\) is faithfully representable.
\end{proof}

The following theorem may be seen as an analogue of Theorem~\ref{thm:surjeqeq}.
\begin{theorem}
  \label{thm:muvak}
  Let \(f_1,\dots,f_k\in A\) be \(\Delta\)\=/homogeneous generators and let \(I \coloneqq \ker(\alpha)\) and \(J_i\coloneqq \ker(\beta_i)\), \(1\leq i\leq m\), with the morphisms \(\alpha\) and \(\beta_i\) as above.
  The set \(\{f_1,\dots,f_k\}\) is a \(\Delta\)\=/homogeneous MUVAK basis of \(A\) with respect to \(v_1,\dots,v_m\) if and only if we have \(J_i^\hom \subseteq I^\hom + \langle t_i\rangle\) for all \(i\in\{1,\dots,m\}\).
\end{theorem}
\begin{proof}
  Assume that \(\{f_1,\dots,f_k\}\) is a \(\Delta\)\=/homogeneous MUVAK basis of \(A\) with respect to \(v_1,\dots,v_m\).
  Fix \(i\in \{1,\dots,m\}\) and let \(h\in J_i\).
  As \(J_i\) is a \(\Delta\)\=/homogeneous ideal, we may assume that \(h\) is \(\Delta\)\=/homogeneous.
  Set \(f\coloneqq h(f_1,\dots,f_k)\).
  By assumption, the polynomial \(f\) is faithfully representable, so by Lemma~\ref{lem:faithrep}, we have \(h^\hom \in I^\hom + \langle\theta\rangle\) with \(\theta = \prod_{j = 1}^mt_j^{v_j(f) - \deg_j(h^\hom)}\).
  By Lemma~\ref{lem:initial}, we have \(\mindeg_i(h) < v_i(f)\), so \(t_i\mid \theta\).
  Hence \(h^\hom\in I^\hom + \langle t_i\rangle\).

  Assume now \(J_i^\hom\subseteq I^\hom + \langle t_i\rangle\) for all \(i\in\{1,\dots,m\}\).
  Let \(f\in A\) be \(\Delta\)\=/homogeneous and pick \(h\in\alpha^{-1}(f)\).
  If \(\mindeg_i(h) \geq v_i(f)\) for all \(i\in\{1,\dots,m\}\), then \(f\) is faithfully representable, so assume that there is \(i\in\{1,\dots,m\}\) with \(\mindeg_i(h) < v_i(f)\).
  Then \(\initial_i(h)\in J_i\) by Lemma~\ref{lem:initial}, so \(\initial_i(h)^\hom\in J_i^\hom\subseteq I^\hom + \langle t_i\rangle\).
  Hence \(f\) is faithfully representable in \(f_1,\dots,f_k\) with respect to \(v_1,\dots,v_m\) by Lemma~\ref{lem:faithrep2}.
\end{proof}

\begin{example}
  \label{ex:muvakrun3}
  This example is part of the series also containing \ref{ex:muvakrun1}, \ref{ex:muvakrun2} and \ref{ex:muvakrun4}.
  Computing the relations of the initial forms gives the ideals \[J_1 = \langle X_3 - X_4, X_1^2 - X_2X_4\rangle\text{ and }J_2 = \langle X_1 - X_2, X_3^2 -X_2X_4\rangle,\] where \(J_1\) corresponds to \(v_x\) and \(J_2\) to \(v_y\).
  Homogenizing yields \[J_1^\hom = \langle X_3 - X_4t_2, X_1^2 - X_2X_4t_2^2\rangle\text{ and }J_2^\hom = \langle X_1 - X_2t_1, X_3^2 - X_2X_4t_1^2\rangle.\]
  One checks that \(J_i^\hom \not\subseteq I^\hom + \langle t_i\rangle\) for \(i = 1,2\).
\end{example}

Theorem~\ref{thm:muvak} motivates Algorithm~\ref{alg:muvak}.
\begin{algorithm}
  \caption{Homogeneous MUVAK basis}
  \label{alg:muvak}
  \Input{\(\Delta\)\=/homogeneous generators \(f_1,\dots,f_k\in A\leq B\), valuations \(v_1,\dots,v_m:B\setminus\{0\}\to \ZZ\)}
  \Output{A \(\Delta\)\=/homogeneous MUVAK basis of \(A\) with respect to \(v_1,\dots,v_m\)}
  \BlankLine
  Initialize \(\mathcal G \coloneqq \{f_1,\dots,f_k\}\)\;
  Compute the ideals \(I^\hom\) and \(J_i^\hom\) for all \(i\in\{1,\dots,m\}\)\;
  \While{\(J_i^\hom\not\subseteq I^\hom + \langle t_i\rangle\) for some \(i\in\{1,\dots,m\}\)}{
    \For{\(i\in\{1,\dots,m\}\)}{
      Write \(J_i^\hom = \langle h_1,\dots,h_r\rangle\) and \(\mathcal G = \{f_1,\dots,f_l\}\)\;
      \(\mathcal G \coloneqq \mathcal G \cup \{h_j^\deh(f_1,\dots,f_l)\mid h_j\notin I^\hom + \langle t_i\rangle,~j = 1,\dots,r\}\)\;
      Update \(I^\hom\) and \(J_i^\hom\)\;
    }
  }
  \Return \(\mathcal G\)\;
\end{algorithm}

\begin{remark}
  We give an efficient method to compute the homogenization of an ideal with respect to the in general only non\-/negative weights \(v_i(f_j)\) in Appendix~\ref{app:bayer}.
\end{remark}

To prove the correctness of Algorithm~\ref{alg:muvak}, we require the following lemma.

\begin{lemma}
  \label{lem:faithreptrans}
  Let \(\{\tilde f_1,\dots,\tilde f_s\}\subseteq A\) be a \(\Delta\)\=/homogeneous MUVAK basis of \(A\) with respect to \(v_1,\dots,v_m\) and let \(f_1,\dots,f_k\in A\) be \(\Delta\)\=/homogeneous generators.
  If \(\tilde f_j\) is faithfully representable in \(f_1,\dots,f_k\) for all \(j\in\{1,\dots,s\}\), then \(\{f_1,\dots,f_k\}\) is a \(\Delta\)\=/homogeneous MUVAK basis for \(A\) with respect to \(v_1,\dots,v_m\).
\end{lemma}
\begin{proof}
  By assumption, there are \(h_j\in K[X_1,\dots,X_k]\) with \(h_j(f_1,\dots,f_k) = \tilde f_j\) and \(\mindeg_i(h_j) = v_i(\tilde f_j)\) for all \(1\leq i \leq m\) and \(1\leq j\leq s\).

  Let \(f\in A\) be \(\Delta\)\=/homogeneous.
  Then \(f\) is faithfully representable in \(\tilde f_1,\dots,\tilde f_s\), so there is \(h\in K[Y_1,\dots,Y_s]\) with \(h(\tilde f_1,\dots,\tilde f_s) = f\) and \(\mindeg_i(h) = v_i(f)\) for all \(i\).
  Put \(\tilde h \coloneqq h(h_1,\dots,h_s)\in K[X_1,\dots,X_k]\), so \(\tilde h(f_1,\dots, f_k) = f\).
  By definition, we have \(\deg_i(Y_j) = v_i(\tilde f_j)\) for all \(i = 1,\dots,m\) and \(j = 1,\dots, s\).
  Hence \(\mindeg_i(\tilde h) \geq \mindeg_i(h)\).
  We conclude \(\mindeg_i(\tilde h) \geq v_i(f)\) for all \(i\), so \(f\) is faithfully representable in \(f_1,\dots,f_k\) as required.
\end{proof}

\begin{proposition}
  \label{prop:muvak}
  Algorithm~\ref{alg:muvak} terminates after finitely many steps if and only if there exists a finite \(\Delta\)\=/homogeneous MUVAK basis of \(A\) with respect to \(v_1,\dots,v_m\).
  If the algorithm terminates, it correctly returns such a basis.
\end{proposition}
\begin{proof}
  Assume that Algorithm~\ref{alg:muvak} terminates with output \(\{f_1,\dots,f_l\}\).
  By construction, these polynomials are generators of \(A\) because the set contains the original generators.
  Further, the \(f_j\) are \(\Delta\)\=/homogeneous as the algorithm only adds homogeneous elements.
  If the algorithm terminates, we must have \(J_i^\hom \subseteq I^\hom + \langle t_i\rangle\) for all \(i\in\{1,\dots,m\}\) for the `updated' ideals corresponding to the \(f_j\).
  Hence \(\{f_1,\dots,f_l\}\) is a \(\Delta\)\=/homogeneous MUVAK basis with respect to \(v_1,\dots,v_m\) by Theorem~\ref{thm:muvak}.
  In particular, \(A\) admits such a finite basis.

  Conversely, assume that \(\{\tilde f_1,\dots,\tilde f_s\}\subseteq A\) is a finite \(\Delta\)\=/homogeneous MUVAK basis with respect to \(v_1,\dots,v_m\).
  If each \(\tilde f_i\), \(1\leq i\leq s\), is faithfully representable in the elements of the set \(\mathcal G\) at some point of the algorithm, then \(\mathcal G\) is a \(\Delta\)\=/homogeneous MUVAK basis of \(A\) with respect to \(v_1,\dots,v_m\) by Lemma~\ref{lem:faithreptrans} and the algorithm terminates by Theorem~\ref{thm:muvak}.
  It hence remains to show that \(\tilde f_1,\dots,\tilde f_s\) are faithfully representable in the elements of the set \(\mathcal G\) after finitely many iterations.

  Let \(i\in\{1,\dots,m\}\) such that at the beginning of the \(i\)\=/th iteration of the \texttt{for}\-/loop there is \(\sigma\in \{1,\dots, s\}\) and \(h\in \alpha^{-1}(\tilde f_\sigma)\) with \(\initial_i(h)^\hom\in J_i^\hom\setminus(I^\hom + \langle t_i\rangle)\).
  If no such \(i\) and \(\sigma\) exist, then \(\tilde f_1, \dots, \tilde f_s\) are faithfully representable in the elements of \(\mathcal G\) by Lemma~\ref{lem:faithrep2}.
  Write \(\mathcal G = \{f_1,\dots,f_l\}\) and \(J_i^\hom = \langle h_1,\dots,h_r\rangle\) as in the algorithm.
  After reordering, we may assume that \(h_1,\dots,h_{r'}\notin I^\hom + \langle t_i\rangle\) and \(h_{r' + 1},\dots,h_r\in I^\hom + \langle t_i\rangle\) for some \(r'\leq r\).
  The algorithm now adds \(\alpha(h_j^\deh)\) for \(j = 1,\dots,r'\) to the set \(\mathcal G\).
  Notice that \(h_j^\deh\in J_i\), so \(\mindeg_i(h_j^\deh) < v_i(\alpha(h^\deh))\) for \(j = 1,\dots,r'\) by Lemma~\ref{lem:initial}.
  Let \(\alpha^+, I^+, J_i^+\) and so on be the updated instances at the end of this iteration of the \texttt{for}\-/loop.
  We claim that \(h_1,\dots,h_r\in (I^+)^\hom + \langle t_i\rangle\).
  Let \(X_{l + 1},\dots,X_{l + r'}\) be the new variables corresponding to \(\alpha(h_1^\deh),\dots,\alpha(h_{r'}^\deh)\) respectively.
  Notice that \(\alpha^+(f') = \alpha(f')\) for any \(f'\in K[X_1,\dots,X_l]\) and in particular \(I\subseteq I^+\) by considering elements of \(K[X_1,\dots,X_l]\) as elements of \(K[X_1,\dots,X_{l + r'}]\) in the canonical way.
  For \(j\in \{1,\dots,r'\}\), we have by construction \(h_j^\deh - X_{l + j}\in I^+\), \(\deg_i(X_{l + j}) > \deg_i(h_j)\) and \(\deg_{i'}(X_{l + j}) \geq \deg_{i'}(h_j)\) for all \(i'\neq i\).
  Hence \((h_j^\deh - X_{l + j})^\hom = h_j - X_{l + j}\theta\) with \(\theta\) a monomial in \(t_1,\dots,t_m\) and \(t_i\mid \theta\).
  So, \(h_j = (h_j^\deh - X_{l + j})^\hom + X_{l + j}\theta \in (I^+)^\hom + \langle t_i\rangle\) as claimed.
  Then also \(\initial_i(h)^\hom\in (I^+)^\hom + \langle t_i\rangle\).
  Hence we may find a preimage \(h'\in\alpha^{-1}(\tilde f_\sigma)\) with \(\mindeg_i(h') > \mindeg_i(h)\) and \(\mindeg_{i'}(h') \geq \mindeg_{i'}(h)\) for \(i'\neq i\) by the same argument as in Lemma~\ref{lem:faithrep2}.
  Continuing with the argument as in Lemma~\ref{lem:faithrep2}, the algorithm will produce generators in which \(\tilde f_1,\dots,\tilde f_s\) are faithfully representable after finitely many steps.
\end{proof}

\begin{example}
  \label{ex:muvakrun4}
  This example is part of the series also containing \ref{ex:muvakrun1}, \ref{ex:muvakrun2} and \ref{ex:muvakrun3}.
  We use Algorithm~\ref{alg:muvak} to compute a MUVAK basis of \(A\).
  Let \(\mathcal G = \{f_1,\dots,f_4\}\).
  With Example~\ref{ex:muvakrun3}, we have the polynomials \(h_1 = X_3 - X_4t_2, h_2 = X_1^2 - X_2X_4t_2^2\notin I^\hom + \langle t_1\rangle\) and \(h_3 = X_1 - X_2t_1\notin I^\hom + \langle t_2\rangle\) (notice that the second generator of \(J_2^\hom\) is indeed an element of \(I^\hom + \langle t_2\rangle\)).
  These give the polynomials \[h_1^\deh(f_1,\dots,f_4) = h_3^\deh(f_1,\dots,f_4) = xy\text{ and }h_2^\deh(f_1,\dots,f_4) = x^4 + 2x^3y.\]
  For simplicity, we only add \(f_5 \coloneqq xy\) to the set of generators.
  Computing the new relations gives the ideals
  \begin{align*}
    I^\hom &= \langle X_3 - X_4t_2 - X_5t_1, X_1 - X_2t_1 - X_5t_2, X_2X_4 - X_5^2\rangle,\\
    J_1^\hom &= \langle X_3 - X_4t_2, X_1 - X_5t_2, X_2X_4 - X_5^2\rangle, \\
    J_2^\hom &= \langle X_3 - X_5t_1, X_1 - X_2t_1, X_2X_4 - X_5^2\rangle,
  \end{align*}
  where the homogenization is now with respect to the matrix \[W = \begin{pmatrix} 1 & 2 & 0 & 0 & 1 \\ 0 & 0 & 1 & 2 & 1 \end{pmatrix}.\]
  We see that now \(J_i^\hom \subseteq I^\hom + \langle t_i\rangle\) and the algorithm terminates with the MUVAK basis \(\{f_1,\dots,f_5\}\) for \(A\).
\end{example}

\section{Examples}
\label{sec:ex}

We give further examples of homogeneous Khovanskii bases before we come to our main application in section~\ref{sec:cox}.
Although everything in this section can be computed by hand, our experiments in the computer algebra system OSCAR \cite{Osc, DEFHJ25} were invaluable to find these examples.

\begin{example}
  \label{ex:nosub}
  We give an example where Algorithm~\ref{alg:homkhov} is unable to find a Khovanskii basis because Algorithm~\ref{alg:homsubduc} does not terminate.
  Let \(K\) be a field, \(B = K[x, y]\) the polynomial ring and \(v:B\setminus\{0\}\to\ZZ\) be the order of divisibility by \(y\) with the natural ordering on \(\ZZ\).
  We have \(\gr_v(B) = B.\)
  Let \(f_1 \coloneqq x, f_2 \coloneqq x + y, f_3 \coloneqq y + y^2\in B\) and \(A = K[f_1, f_2, f_3]\).
  The grading is unimportant for this example and we just put \(\Delta \coloneqq \{0\}\).
  Clearly, \(A = B\) and a Khovanskii basis of \(A\) with respect to \(v\) is given by \(\{x, y\}\).

  If we run Algorithm~\ref{alg:homkhov}, we compute:
  \[\begin{array}{c|c}
    & \initial_v^{\{0\}} \\
    \hline
    f_1 & x \\
    f_2 & x \\
    f_3 & y
  \end{array}\]
  The algorithm then calls the subduction algorithm to reduce \(f_4\coloneqq f_1 - f_2 = -y\) by \(\{f_1,f_2, f_3\}\).
  This will enter an infinite loop by replacing \(y\) by \(y^2\), \(y^3\) and so on.

  Let us run this example with Algorithm~\ref{alg:muvak} and compute a homogeneous MUVAK basis with respect to \(v\).
  In the first iteration, we compute the ideals \[I^\hom = \langle X_1^2 - 2X_1X_2 - X_1 + X_2^2 + X_2 - X_3t\rangle\text{ and }J^\hom = \langle X_1 - X_2\rangle.\]
  Because \(X_1 - X_2 \notin I^\hom + \langle t\rangle\), the algorithm adds the polynomial \(f_4 = f_1 - f_2 = -y\) to the generating set.
  The updated ideals are now \[I^\hom = \langle X_1 - X_2 - X_4t, X_3 + X_4 - X_4^2t \rangle\text{ and }J^\hom = \langle X_1 - X_2, X_3 + X_4\rangle\] and one sees that \(J^\hom \subseteq I^\hom + \langle t \rangle\) as is required for the termination of the algorithm.
  So, the algorithm produced the MUVAK basis \(\{f_1,f_2, f_3,f_4\}\) with respect to \(v\).
\end{example}

The next examples show how the cardinality of a homogeneous Khovanskii basis depends on the chosen grading.
Notice that any subgroup of \(\Delta\) is a normal subgroup as \(\Delta\) is abelian.
For a subgroup \(\Theta\leq\Delta\), we endow \(A\) with a `coarsened' grading by \(\Theta\), that is, \[A_\theta \coloneqq \bigoplus_{\delta\in[\theta]}A_\delta\] for all \(\theta\in \Theta\), where \([\theta]\) denotes the class of \(\theta\) in \(\Delta/\Theta\).
\begin{lemma}
  Let \(\Theta\leq \Delta\) be a subgroup and let \(\{f_1,\dots,f_k\}\) be a \(\Theta\)\=/homogeneous Khovanskii basis of \(A\) with respect to \(v\).
  If \(f_1,\dots,f_k\) are \(\Delta\)\=/homogeneous, then \(\{f_1,\dots,f_k\}\) is a \(\Delta\)\=/homogeneous Khovanskii basis of \(A\) with respect to \(v\).
\end{lemma}
\begin{proof}
  This follows directly from Theorem~\ref{thm:surjeqeq}.
  Write \(J_\Delta\), respectively \(J_\Theta\), for the ideal \(J\) coming from the respective grading.
  By construction, we have \(J_\Delta \subseteq J_\Theta\) (any \(\Delta\)\=/homogeneous relations must be \(\Theta\)\=/homogeneous).
  The ideal \(\initial_\Gamma(I)\) in the theorem is independent of the grading by \(\Delta\) or \(\Theta\) and we have by assumption \(\initial_\Gamma(I) = J_\Theta\) as well as \(\initial_\Gamma(I) \subseteq J_\Delta\).
  Hence \(\initial_\Gamma(I) = J_\Delta\) and Theorem~\ref{thm:surjeqeq} gives the claim.
\end{proof}

\begin{example}
  \label{ex:gradC2}
  Let \(K\) be a field of characteristic different from 2 and let \(B = K[x, y]\).
  We endow \(B\) with a grading via the weight vector \(w = (1, 0)\in\ZZ^2\), so we put \(\deg_w(x) \coloneqq 1\) and \(\deg_w(y) \coloneqq 0\).
  As valuation \(v:B\setminus\{0\}\to \ZZ\), we consider the \emph{grading function} (see \cite[Example~2.2]{KM19}) of this grading, that is, for \(f\in B\setminus\{0\}\) we set \(v(f) = \min\{\deg_w(t)\mid t\in \Supp(f)\}\).
  In particular, we have \(\gr_v(B) \cong B\).

  Let \(f_1 \coloneqq x^2 + y^2\), \(f_2 \coloneqq x^2 - y^2\) and \(f_3 \coloneqq xy\).
  We want to compute a \(\Delta\)\=/homogeneous Khovanskii basis for the subalgebra \(A \coloneqq K[f_1,f_2,f_3]\leq B\) which is graded by \(\Delta \coloneqq (\ZZ/2\ZZ)^2\times \ZZ\) via \(\deg_\Delta(f_1) = (0, 0, 2)\), \(\deg_\Delta(f_2) = (1, 0, 2)\) and \(\deg_\Delta(f_3) = (0, 1, 2)\).
  By Proposition~\ref{prop:deltafindim}, Algorithm~\ref{alg:homkhov} is applicable because the grading refines the standard grading.

  To compute the initial forms, write \(K\Delta = K[g_1,g_2, g_3]/\langle g_1^2 - 1, g_2^2 - 1\rangle\).
  Then we have
  \[\def\arraystretch{1.2}\begin{array}{c|c|c}
    & \initial_v & \initial_v^\Delta \\
    \hline
    f_1 & y^2 & y^2g_3^2 \\
    f_2 & -y^2 & -y^2g_1g_3^2 \\
    f_3 & xy & xyg_2g_3^2
  \end{array}\]
  The kernel of the morphism \[K[X_1,X_2,X_3]\to B\otimes_KK\Delta,\ X_i\mapsto\initial_v^\Delta(f_i)\] is given by \(J = (X_1^2 - X_2^2)\).
  Hence the algorithm would check the potential new element \(f_1^2 - f_2^2 = 4x^2y^2\) with \(\initial_v^\Delta(4x^2y^2) = 4x^2y^2\) and the subduction algorithm rewrites this as \(f_1^2 - f_2^2 = 4f_3^2\).
  So, \(\{f_1,f_2,f_3\}\) is a \(\Delta\)\=/homogeneous Khovanskii basis of \(A\).
\end{example}

\begin{example}
  \label{ex:gradZ}
  Let us consider the same algebras \(A\) and \(B\) and the same valuation \(v\) as in Example~\ref{ex:gradC2}, but this time we take only the standard grading of \(A\) (and \(B\)) into account and compute a \(\ZZ\)\=/homogeneous Khovanskii basis of \(A\) with respect to \(v\).
  We will see that the polynomials \(f_1,f_2,f_3\) do not form a \(\Delta\)\=/homogeneous Khovanskii basis in this case.

  We write \(K\ZZ = K[g]\) for the group ring and compute as above:
  \[\def\arraystretch{1.2}\begin{array}{c|c|c}
    & \initial_v & \initial_v^\ZZ \\
    \hline
    f_1 & y^2 & y^2g^2 \\
    f_2 & -y^2 & -y^2g^2 \\
    f_3 & xy & xyg^2
  \end{array}\]
  In this case, the kernel of \[K[X_1,X_2,X_3]\to B\otimes_KK\Delta,\ X_i\mapsto\initial_v^\Delta(f_i)\] is given by \(J = (X_1 + X_2)\).
  The subduction algorithm cannot reduce the element \(f_4 \coloneqq f_1 + f_2 = 2x^2\) with \(\initial_v^\ZZ(f_4) = 2x^2g^2\), so we add \(f_4\) to the basis.

  We now compute \(J = (X_1 + X_2, X_2X_4 + 2X_3^2)\) as kernel of \[K[X_1,\dots,X_4]\to B\otimes_KK\Delta,\ X_i\mapsto\initial_v^\Delta(f_i).\]
  Evaluating the new generator gives \(f_5 \coloneqq f_2f_4 + 2f_3^2 = 2x^4\) with \(\initial_v^\ZZ(f_5) = 2x^4g^4\) and the subduction algorithm finds the representation \(f_5 = \frac{1}{2}f_4^2\).
  We conclude that \(\{f_1,\dots, f_4\}\) is a \(\ZZ\)\=/homogeneous Khovanskii basis for \(A\) and \(v\).
\end{example}

\begin{example}
  \label{ex:inf}
  The situation described in Examples~\ref{ex:gradC2} and \ref{ex:gradZ} can be even more jarring:
  We now give an example for a subalgebra which has a finite homogeneous Khovanskii basis with respect to a certain grading, but does not have a finite Khovanskii basis with  respect to another one.

  For this, we adapt \cite[Example~6.6.7]{KR05}.
  There, Kreuzer and Robbiano show that the subalgebra generated by \(x_1 + x_2, x_1x_2, x_1x_2^2\in K[x_1,x_2]\) has no finite SAGBI basis, no matter which monomial ordering is used on \(K[x_1, x_2]\).
  We choose a slightly different set of polynomials to produce the situation we aim for.
  Let \(K\) be a field and let \(B = K[x, y, z]\).
  Let \(>_1\) be the negative lexicographical ordering on \(z\), that is, induced by \(1 >_1 z\), and let \(>_2\) be a monomial ordering on the variables \(x, y\) with \(x >_2 y\).
  Let now \(>\;\coloneqq (>_1, >_2)\) be the block ordering, which gives a total ordering on all monomials of \(B\).
  Note that \(>_1\) is a local ordering (in the terminology of \cite{GP08}) and so \(>\) is out of scope of the orderings considered in \cite{KR05} for SAGBI bases.
  Nevertheless, there is a valuation \(v:B\setminus\{0\}\to\ZZ^3\) induced by \(>\) as in Remark~\ref{rem:sagbi} and we can ask for Khovanskii bases with respect to \(v\).
  We again have \(\gr_v(B) \cong B\).
  Let finally \(f_1 \coloneqq x + y + z\), \(f_2\coloneqq xy\) and \(f_3 \coloneqq xy^2\) and put \(A\coloneqq K[f_1,f_2,f_3]\leq B\).
  The polynomials \(f_1\), \(f_2\) and \(f_3\) are algebraically independent and we may endow \(A\) with a grading by \(\Delta\coloneqq\ZZ^3\) by putting \(\deg_\Delta(f_i) = e_i\), where \(e_i\in \ZZ^3\) denotes the element with entry 1 in position \(i\) and 0 everywhere else.
  Then we have the group ring \(K\Delta = K[g_1, g_2, g_3]\) and one computes:
  \[\def\arraystretch{1.2}\begin{array}{c|c|c}
    & \initial_v & \initial_v^\Delta \\
    \hline
    f_1 & x & xg_1 \\
    f_2 & xy & xyg_2 \\
    f_3 & xy^2 & xy^2g_3
  \end{array}\]
  Hence there is no relation between the \(\initial_v^\Delta(f_i)\), \(i = 1, 2, 3\), and a \(\Delta\)\=/homogeneous Khovanskii basis of \(A\) with respect to \(v\) is given by \(\{f_1,f_2,f_3\}\) (Algorithm~\ref{alg:homkhov} is applicable by Proposition~\ref{prop:deltafindim}).

  We now consider the standard grading by \(\ZZ\) on \(A\) and prove in the following that there is no finite \(\ZZ\)\=/homogeneous Khovanskii basis of \(A\).
  Let \[\rho:B\to K[x, y, z]/\langle z\rangle \cong K[x, y]\] be the canonical projection.
  Then \(\rho(A)\) is the algebra considered in \cite[Example~6.6.7]{KR05}.
  Further, by construction of \(>\), we have \(\LM_{>_2}(\rho(f)) = \rho(\LM_>(f))\).
  Hence \(\rho\) induces a surjection \(\initial_v^\ZZ(A) \to \initial_{v_2}^\ZZ(\rho(A))\) where \(v_2\) is the valuation on \(K[x, y]\) induced by \(>_2\).
  Kreuzer and Robbiano prove that there is an isomorphism \(\initial_{v_2}^\ZZ(\rho(A)) \cong K[x, xy, xy^2, \dots]\) and conclude that \(\initial_{v_2}^\ZZ(\rho(A))\) is not a finitely generated algebra.
  Precisely, the monomial \(xy^d\) is not in the algebra \(K[x, xy, \dots, xy^{d - 1}]\) for every \(d\geq 1\).
  The same holds true for \(\initial_v^\ZZ(\rho(A))\) because the monomials in \(\initial_v^\ZZ(\rho(A))\) which are not in \(\initial_{v_2}^\ZZ(\rho(A))\) are all divisible by \(z\).
  In conclusion, there is no finite \(\ZZ\)\=/homogeneous Khovanskii basis for \(A\) with respect to \(v\).
\end{example}

\section{Cox rings of minimal models of quotient singularities}
\label{sec:cox}

We describe an application of homogeneous MUVAK bases, namely the computation of \emph{Cox rings} of certain varieties.
The Cox ring of a toric variety was introduced by Cox \cite{Cox95} and then carried over to the setting of birational geometry by Hu and Keel \cite{HK00}.
The definition of the Cox ring \(\mathcal R(X)\) of a normal variety \(X\) is quite involved.
We just point out that we have \[\mathcal R(X) = \bigoplus_{[D]\in\Cl(X)}\Gamma(X,\mathcal O_X(D)),\] if the class group \(\Cl(X)\) of \(X\) is free.
See \cite[Section~1.4]{ADHL15} for the general case and a profound reference for the topic in general.
In the following, we are interested in Cox rings arising in the context of quotient singularities.

\subsection{Quotient singularities}
\label{subsec:quotsing}
Throughout this section, let \(V\) be a finite\-/dimensional vector space over \(\CC\) and let \(G\leq\SL(V)\) be a finite group.
We write \(V/G\coloneqq \Spec\CC[V]^G\) for the linear quotient of \(V\) by \(G\), where \(\CC[V]^G\) is the invariant ring of \(G\).
By the classical theorem of Chevalley--Serre--Shephard--Todd \cite[Théorème~1']{Ser68}, the variety \(V/G\) is singular and we hence refer to it as a \emph{quotient singularity}.

We are interested in the birational geometry of \(V/G\).
In the following, we give a terse summary of the steps that lead us to a certain Cox ring; we refer the reader to the given references for more details.
By deep results of the minimal model programme \cite{BCHM10}, there exists a \emph{minimal model}, or more precisely a \emph{\(\QQ\)\=/factorial terminalization}, \(X\to V/G\) of \(V/G\), see \cite[Theorem~2.1.15]{Sch23} for details of how this follows from \cite[Corollary~1.4.3]{BCHM10}.
Further, this minimal model turns out to be a \emph{(relative) Mori dream space} \cite[Theorem~3.4.10]{Gra19}.
Algebraically, this means that the Cox ring \(\mathcal R(X)\) is a finitely generated \(\CC\)\=/algebra.
This makes it accessible from an algorithmic point of view.

Our interest in computing the ring \(\mathcal R(X)\) is motivated by the aim to have an `algorithmic minimal model programme':
The idea is to first compute \(\mathcal R(X)\) \emph{without} prior knowledge of \(X\) and then recover \(X\), and in fact any minimal model of \(V/G\), from \(\mathcal R(X)\).
See \cite[Section~7]{Sch23} for a proof of concept of this idea.
The setting of quotient singularities appears to be a natural first step to implement the ideas discussed in a much more general context in \cite{Laz24} because of the additional structure coming from the group action.

In \cite{Yam18, Gra19}, Yamagishi and Grab construct generators of \(\mathcal R(X)\) via a homogeneous MUVAK basis of \(\mathcal R(V/G)\), where \(X\to V/G\) is a minimal model of a quotient singularity as before.
We now recall this result in the language of this article.

To describe the setting, we start with the Cox ring \(\mathcal R(V/G)\) of \(V/G\).
By a theorem of Arzhantsev--Ga\v{\i}fullin \cite[Theorem~3.1]{AG10}, we have \(\mathcal R(V/G)\cong \CC[V]^{[G,G]}\), where \([G,G]\) denotes the commutator subgroup of \(G\).
The ring \(\mathcal R(V/G)\) is graded by the class group \(\Cl(V/G)\) and we have \(\Cl(V/G) = \Hom(G, \CC^\times)\) by \cite[Theorem~3.9.2]{Ben93}.
Let \(\Ab(G)\coloneqq G/[G,G]\) be the abelianization of \(G\) and write \(\Ab(G)^\vee\) for the group of characters of \(\Ab(G)\).
By elementary character theory, there is an isomorphism \(\Hom(G,\CC^\times)\cong \Ab(G)^\vee\).
The isomorphism \(\mathcal R(V/G)\cong \CC[V]^{[G,G]}\) is \(\Ab(G)^\vee\)\=/graded, where the grading on \(\CC[V]^{[G,G]}\) is coming from the action of \(\Ab(G)\), that is, \[\CC[V]^{[G,G]}_\chi \coloneqq \{f\in \CC[V]^{[G,G]}\mid \gamma.f = \chi(\gamma)f \text{ for all }\gamma\in \Ab(G)\}\] for every \(\chi\in\Ab(G)^\vee\).

The ring \(\mathcal R(X)\) is connected to \(\mathcal R(V/G)\) via the injective morphism \[\Theta:\mathcal R(X) \to \mathcal R(V/G)\otimes_\CC\CC[\Cl(X)^{\mathrm{free}}] \cong \big(\CC[V]^{[G,G]}\big)[t_1^{\pm1},\dots,t_m^{\pm 1}],\] where \(\Cl(X)^{\mathrm{free}}\) is the free part of the class group of \(X\) and \(m\in\ZZ_{\geq 0}\) the rank of this group, so \(\Cl(X)^{\mathrm{free}} \cong \ZZ^m\) \cite[Proposition~4.1.5]{Gra19}.
The idea to study the morphism \(\Theta\) originated to our knowledge in \cite{Don16} and was extended in \cite{DG16}, \cite{DW17}, \cite{Yam18} and \cite{Gra19}; see also \cite[Section~6.1]{Sch23} for a more detailed overview.

We can now relate generators of \(\mathcal R(X)\) to a homogeneous MUVAK basis of \(\CC[V]^{[G,G]}\).
For this, there are \(m\) valuations \(v_1,\dots,v_m:\CC[V]\setminus\{0\}\to \ZZ\) that are defined solely by the action of \(G\) on \(V\) and originate in the context of McKay correspondence, see \cite{IR96}.

The following theorem appears in \cite[Proposition~4.4]{Yam18} and \cite[Theorem~4.1.15]{Gra19}.
The integers \(r_j\in\ZZ_{\geq 0}\) in the theorem are the orders of certain elements of \(G\) (\emph{junior elements}) corresponding to the valuations \(v_j\), see again \cite{IR96} for details.
\begin{theorem}[Yamagishi, Grab]
  Let \(f_1,\dots, f_k\in\CC[V]^{[G,G]}\) be \(\Ab(G)^\vee\)\=/homogeneous generators.
  Then the Cox ring \(\mathcal R(X)\) is generated by the preimages of
  \[f_i\otimes\prod_{j = 1}^mt_j^{v_j(f_i)},\ 1\leq i\leq k, \text{ and }1\otimes t_j^{-r_j},\ 1\leq j\leq m,\]
  under \(\Theta\) if and only if \(\{f_1,\dots, f_k\}\) is a \(\Ab(G)^\vee\)\=/homogeneous MUVAK basis of \(\CC[V]^{[G,G]}\) with respect to \(v_1,\dots,v_m\).
\end{theorem}

\subsection{Comparison to Yamagishi's algorithm}
\label{subsec:comp}

By \cite[Corollary~4.13]{Yam18}, a finite \(\Ab(G)^\vee\)\=/homogeneous MUVAK basis of \(\CC[V]^{[G,G]}\) with respect to \(v_1,\dots,v_m\) exists.
Hence, by Proposition~\ref{prop:muvak}, Algorithm~\ref{alg:muvak} is applicable to compute such a basis.

If there is only one valuation \(v = v_1\), then a MUVAK basis is a Khovanskii basis by Theorem~\ref{thm:muvakkhov} and we may also use Algorithm~\ref{alg:homkhov} by virtue of Proposition~\ref{prop:homkhovcorrect} assuming that the subduction algorithm always terminates.
This is indeed the case: by the following Remarks~\ref{rem:stdhom1} and \ref{rem:stdhom2}, we may replace the grading group \(\Ab(G)^\vee\) by \(\Ab(G)^\vee\times \ZZ\), where the graded components are exactly the intersections of the components with respect to the grading by \(\Ab(G)^\vee\) and the grading by the standard degree.
This extension of the grading group guarantees the termination of the subduction algorithm by Proposition~\ref{prop:deltafindim}.

\begin{remark}
  \label{rem:stdhom1}
  Notice that \(\CC[V]^{[G,G]}\) is a standard graded ring and that the action of \(\Ab(G)\) on \(\CC[V]^{[G,G]}\) is linear, so preserves the standard degree.
  In consequence, the grading by \(\Ab(G)^\vee\) is well\-/behaved in the sense that we may decompose any \(\Ab(G)^\vee\)\=/homogeneous polynomial \(f\in \CC[V]^{[G,G]}\) as a sum of standard homogeneous polynomials which are \(\Ab(G)^\vee\)\=/homogeneous as well.
\end{remark}

\begin{remark}
  \label{rem:stdhom2}
  Given two standard homogeneous polynomials \(f_1, f_2\in \CC[V]\), we have \[v(f_1 + f_2) = \min\{v(f_1), v(f_2)\}\] by construction of the valuation \(v\), see \cite{IR96}.
\end{remark}

In \cite{Yam18}, Yamagishi gives an algorithm that computes a \(\Ab(G)^\vee\)\=/homogeneous MUVAK basis of \(\CC[V]^{[G,G]}\) with respect to \(v_1,\dots,v_m\).
In broad strokes, Yamagishi's algorithm first adds generators until the set of generators is a MUVAK basis for every valuation \(v_i\) \emph{individually}, that is, we have generators \(f_1,\dots,f_k\in\CC[V]^{[G,G]}\) such that for all \(f\in \CC[V]^{[G,G]}\) there are polynomials \(h_1,\dots, h_m\in\CC[X_1,\dots, X_k]\) with \(h_i(f_1,\dots,f_k) = f\) and \(v_i(t(f_1,\dots,f_k)) \geq v_i(f)\) for all \(t\in \Supp(h_i)\) and \(1\leq i\leq m\).
Afterwards, the algorithm considers pairs, triples, and so on, of valuations to ensure that eventually one has a MUVAK basis with respect to \(v_1,\dots,v_m\), that is, that for all \(f\in\CC[V]^{[G,G]}\) there is \(h\in \CC[X_1,\dots, X_k]\) with \(h(f_1,\dots,f_k) = f\) and \(v_i(t(f_1,\dots,f_k)) \geq v_i(f)\) for all \(t\in \Supp(h)\) and \(1\leq i\leq m\).

\begin{table}[bp]
  \caption{Runtimes of computing a \(\Ab(G)^\vee\)\=/homogeneous MUVAK basis of \(\CC[V]^{[G,G]}\) with respect to the valuations arising from McKay correspondence (all times in seconds)}
  \label{tab:timings}
  \begin{tabular}{l|c|S|S|S}
    Group & Number of & {Yamagishi's} & {Khovanskii basis} & {MUVAK basis} \\
    & valuations & {algorithm \cite{Yam18}} & {(Algorithm~\ref{alg:homkhov})} & {(Algorithm~\ref{alg:muvak})} \\
    \hline
    \(G(2, 1, 2)^\circledast\)   & 2 & \tablenum{0.13}  & {n/a}           & \tablenum{0.10} \\
    \(G(3, 1, 2)^\circledast\)   & 3 & \tablenum{1.30}  & {n/a}           & \tablenum{0.87} \\
    \(G(3, 3, 2)^\circledast\)   & 1 & \tablenum{0.15}  & \tablenum{0.14} & \tablenum{0.18} \\
    \(G(4, 1, 2)^\circledast\)   & 4 & \tablenum{10.9}  & {n/a}           & \tablenum{5.0} \\
    \(G(4, 2, 2)^\circledast\)   & 3 & \tablenum{0.28}  & {n/a}           & \tablenum{0.23} \\
    \(G(4, 4, 2)^\circledast\)   & 2 & \tablenum{0.15}  & {n/a}           & \tablenum{0.11} \\
    \(G(5, 1, 2)^\circledast\)   & 5 & \tablenum{180}   & {n/a}           & \tablenum{21.0} \\
    \(G(5, 5, 2)^\circledast\)   & 1 & \tablenum{5.37}  & \tablenum{1.68} & \tablenum{4.12} \\
    \(G(6, 1, 2)^\circledast\)   & 6 & \tablenum{1165}  & {n/a}           & \tablenum{107} \\
    \(G(6, 2, 2)^\circledast\)   & 4 & \tablenum{2.62}  & {n/a}           & \tablenum{2.10} \\
    \(G(6, 3, 2)^\circledast\)   & 2 & \tablenum{106}   & {n/a}           & \tablenum{4.39} \\
    \(G(6, 6, 2)^\circledast\)   & 2 & \tablenum{0.53}  & {n/a}           & \tablenum{0.34} \\
    \(G(7, 1, 2)^\circledast\)   & 7 & \tablenum{29253} & {n/a}           & \tablenum{2322} \\
    \(G(7, 7, 2)^\circledast\)   & 1 & \tablenum{145}   & \tablenum{4.89} & \tablenum{153} \\
    \(G(8, 2, 2)^\circledast\)   & 5 & \tablenum{31.7}  & {n/a}           & \tablenum{16.9} \\
    \(G(8, 4, 2)^\circledast\)   & 3 & \tablenum{13.4}  & {n/a}           & \tablenum{5.11} \\
    \(G(8, 8, 2)^\circledast\)   & 2 & \tablenum{2.94}  & {n/a}           & \tablenum{1.34} \\
    \(G(9, 9, 2)^\circledast\)   & 1 & \tablenum{4748}  & \tablenum{14.9} & \tablenum{5978} \\
    \(G(10, 10, 2)^\circledast\) & 2 & \tablenum{18.4}  & {n/a}           & \tablenum{3.51} \\
    \(G_4^\circledast\)          & 2 & \tablenum{50.0}  & {n/a}           & \tablenum{12.5} \\
    \(G_5^\circledast\)          & 4 & \tablenum{287}   & {n/a}           & \tablenum{109} \\
    \(G_6^\circledast\)          & 3 & \tablenum{492}   & {n/a}           & \tablenum{42.9} \\
    \(G_7^\circledast\)          & 5 & \tablenum{970}   & {n/a}           & \tablenum{495} \\
    \(G_{12}^\circledast\)       & 1 & \tablenum{>36000}& \tablenum{415}  & \tablenum{>36000}\\
  \end{tabular}
\end{table}

We implemented the algorithms in this paper in the computer algebra system OSCAR \cite{Osc, DEFHJ25}.
This implementation is freely available online at \begin{center}\url{https://gitlab.com/math5724907/homogeneouskhovanskii}.\end{center}
We also implemented Yamagishi's algorithm as part of \cite{Sch23} and this is to our knowledge the only available implementation of the algorithm in \cite{Yam18}.
We now compare the performance of the different algorithms.
For this, we ran the algorithms with certain symplectic reflection groups \(G\leq\GL(V)\) \cite{Coh80}.
The corresponding linear quotients \(V/G\) are of particular interest as they are examples of \emph{symplectic singularities} \cite{Bea00}, see \cite[Chapter 2]{Sch23} for an overview and more references.
A large class of symplectic reflection groups is related to complex reflection groups: given a complex reflection group \(H\leq \GL(W)\), where \(W\) is a finite\-/dimensional complex vector space, the group \[H^\circledast \coloneqq \left\{\!\begin{psmallmatrix} h & \\ & (h^{-1})^\top\!\end{psmallmatrix}\! \ \middle|\ h\in H\right\}\leq \GL(W \oplus W^\ast)\] is a symplectic reflection group.
For a symplectic reflection group, the number of valuations arising from McKay correspondence coincides with the number of conjugacy classes of symplectic reflections.

The runtimes of the different algorithms are presented in Table~\ref{tab:timings}, where for every timing we took the minimum of three runs.
In the table, we label the groups according to the classification of complex reflection groups by Shephard and Todd \cite{ST54}.
We used the matrix models from CHEVIE \cite{Mic15}, as constructed in \cite{MM10}, for these groups.
Recall that Algorithm~\ref{alg:homkhov} is only applicable if there is just one valuation.

We see that in the cases, where there is more than one valuation, Algorithm~\ref{alg:muvak} is faster than Yamagishi's algorithm, sometimes up to a factor of 10.
When there is just one valuation, the computations carried out by these two algorithms largely agree and hence also their runtimes are comparable.
In these cases, Algorithm~\ref{alg:homkhov} is applicable and faster than both of the other algorithms.

\appendix
\section{A generalization of Bayer's method}
\label{app:bayer}

In Algorithm~\ref{alg:muvak}, we need to compute the homogenization of an ideal with respect to non\-/negative weights.
This involves the computation of a saturation, for which a naive approach requires potentially several expensive Gröbner basis computations.
In the following, we describe how one may adapt an algorithm known as `Bayer's method' to compute such a saturation, and hence the homogenization, efficiently.
The results of this appendix were already published as part of \cite{Sch23}; we repeat them here for the reader's convenience.

We require some notions from the theory of Gröbner bases (or standard bases) and refer to \cite[Chapter~1]{GP08} for the basic definitions.
In particular, if \(>\) is a monomial ordering on a polynomial ring \(K[X_1,\dots, X_k]\), then following \cite[Definition~1.2.1]{GP08} we allow that \(X_i < 1\), which is occasionally excluded in the definition of monomial orderings.
We also adopt the terminology of only speaking of `Gröbner bases' if the ordering is global and of `standard bases' in general, see \cite[Definition~1.6.1]{GP08}.

For the following discussion, let \(K[X_1,\dots, X_k]\) be graded by an integral weight vector \(\mathbf w = (w_1,\dots, w_k)\in\ZZ^k_{\geq 0}\) via \(\deg_{\mathbf w}(X_i) = w_i\) and let \(I\ideal K[X_1,\dots, X_k]\) be an ideal.
We add an additional variable \(t\) to \(K[X_1,\dots, X_k]\) and want to compute the homogenization \(I^\hom\) of \(I\) with respect to the grading \(\deg_{\mathbf w}\) and the variable \(t\).
Depending on whether we want to homogenize `positively' or `negatively', we set \(\deg_{\mathbf w}(t) \coloneqq 1\) or \(\deg_{\mathbf w}(t) \coloneqq -1\), respectively.

If \(w_i\neq 0\) for all \(1\leq i\leq k\), there is a quite simple method to compute \(I^\hom\) that only involves the computation of a Gröbner basis of \(I\) with respect to the weighted degree ordering defined by \(\mathbf w\), see \cite[Exercise~1.7.5]{GP08}.
However, although the weights in our application are non\-/negative, they might in general be zero, so they do not give rise to a total ordering on the set of monomials and we cannot make use of this approach.

A more general idea for the computation of \(I^\hom\) is to homogenize a set of generators of \(I\) resulting in an ideal \(\tilde I\) and then to compute the saturation of \(\tilde I\) with respect to \(t\) as one has \(I^\hom = \tilde I : \langle t\rangle^\infty\) by \cite[Corollary~4.3.8]{KR05}.
However, a naive computation of this saturation potentially involves several expensive Gröbner basis computations as one iteratively computes ideal quotients until the result stabilizes, see \cite[Section~1.8.9]{GP08}.

Our more specialized approach for the saturation is based on `Bayer's method', see \cite[p.\ 120]{Bay82}, \cite[Proposition~5.1.11]{Sti05}.
In a nutshell, this means that one computes a standard basis of \(\tilde I\) with respect to a tailored monomial ordering and then only needs to divide the elements of this basis by \(t\), see Proposition~\ref{prop:bayer} for the precise statement.
The core idea of Bayer's method is the following observation.
Let \(f\in K[X_1,\dots, X_k]\) be a homogeneous polynomial with respect to the standard grading.
Then the leading term \(\LT(f)\) with respect to the degree reverse lexicographical ordering \cite[Example~1.2.8~(1)~(ii)]{GP08} is divisible by \(X_k\) if and only if \(f\) is divisible by \(X_k\).
We now translate this to the grading \(\deg_{\mathbf w}\) by considering a certain matrix ordering.

We assume that \(\mathbf w \neq 0\), so after reordering the variables we may assume \(w_k \neq 0\).
If we have \(\mathbf w = 0\), then any polynomial and hence any ideal is homogeneous, so the computation of the homogenization is trivial.
Let \[M \coloneqq \begin{pmatrix}
  w_1 & \cdots & \cdots & w_k    & \deg_{\mathbf w}(t) \\
  0   & \cdots & \cdots & 0      & -1     \\
  1   & \ddots &        & \vdots & 0      \\
      & \ddots & \ddots & \vdots & \vdots \\
  0   &        & 1      & 0      & 0
\end{pmatrix}\in\ZZ^{(k + 1)\times(k + 1)}.\]
This is a matrix of full rank as \(w_k \neq 0\) and hence induces a monomial ordering \(>_M\) on the monomials of \(K[X_1,\dots, X_k, t]\) by multiplying the exponent vectors by \(M\) (from the left) and then using the lexicographic ordering on \(\ZZ^{k + 1}\), see \cite[Remark 1.2.7]{GP08}.

One directly convinces oneself of the following facts.

\begin{lemma}
  \label{lem:bayer}
  Let \(f_1 = X_1^{a_1}\cdots X_k^{a_k}t^{a_{k + 1}}\) and \(f_2 = X_1^{b_1}\cdots X_k^{b_k}t^{b_{k + 1}}\) be two monomials.
  Then we have:
  \begin{enumerate}
    \item\label{lem:bayer:a} \(X_i >_M 1\) for \(1\leq i\leq k\);
    \item\label{lem:bayer:b} if \(\deg_{\mathbf w}(t) = 1\), then \(t >_M 1\), and \(1 >_M t\) otherwise;
    \item\label{lem:bayer:c} if \(\deg_{\mathbf w}(f_1) > \deg_{\mathbf w}(f_2)\), then \(f_1 >_M f_2\);
    \item\label{lem:bayer:d} if \(\deg_{\mathbf w}(f_1) = \deg_{\mathbf w}(f_2)\) and \(f_1 >_M f_2\), then \(a_{k + 1}\leq b_{k + 1}\).
  \end{enumerate}
\end{lemma}

Point \ref{lem:bayer:d} is the direct generalization of the above mentioned `core idea' for Bayer's method: for a \(\deg_{\mathbf w}\)\=/homogeneous polynomial \(f\), we have \(t\mid \LT_{>_M}(f)\) if and only if \(t\mid f\), where \(\LT_{>_M}\) is the leading term with respect to \(>_M\).
It follows from point \ref{lem:bayer:a} that \(>_M\) is global with respect to the variables \(X_1,\dots, X_k\) and this also extends to the variable \(t\), if \(\deg_{\mathbf w}(t) = 1\), that is, if we homogenize positively, by point \ref{lem:bayer:b}.
However, if \(\deg_{\mathbf w}(t) = -1\), then \(t <_M 1\), so the ordering \(>_M\) is local with respect to \(t\).
This second case is more challenging:
in order to speak about a standard basis of \(\tilde I\) with respect to \(>_M\), we have to consider the extension of \(\tilde I\) to the localization \[K[X_1,\dots, X_k, t]_{>_M} \coloneqq S^{-1}K[X_1,\dots, X_k, t],\] where \[S \coloneqq \{u\in K[X_1,\dots, X_k, t]\setminus \{0\}\mid \LT_{>_M}(u)\text{ is constant}\}.\]
See \cite[Section~1.5]{GP08} for details.
If \(\deg_{\mathbf w}(t) = -1\), we have \(S = \{h\in K[t]\mid h(0)\neq 0\}\), so we may identify \[K[X_1,\dots, X_k, t]_{>_M} = \big(K[t]_{(t)}\big)[X_1,\dots, X_k].\]
In case \(\deg_{\mathbf w}(t) = 1\), the ordering \(>_M\) is global, so we have \(S = K^\times\) by \cite[p.\ 39]{GP08}.

\begin{lemma}
  \label{lem:extconteq}
  With the above notation, let \(J\ideal K[X_1,\dots, X_k, t]\) be a \(\deg_{\mathbf w}\)\=/homogeneous ideal.
  Then we have \[J = (S^{-1}J)\cap K[X_1,\dots, X_k, t].\]
\end{lemma}
\begin{proof}
  If \(\deg_{\mathbf w}(t) = 1\), there is nothing to show, so let \(\deg_{\mathbf w}(t) = -1\).
  For a polynomial \(f\in(S^{-1}J)\cap K[X_1,\dots, X_k, t]\), there is \(u\in S\) with \(uf\in J\).
  Writing \(u = \sum_ja_jt^j\) with \(a_j\in K\), we have \(a_jt^jf\in J\) for all \(j\) since \(J\) is homogeneous.
  But \(a_0 \neq 0\) by assumption, so \(f\in J\).
\end{proof}

\begin{proposition}[Bayer's method]
  \label{prop:bayer}
  Let \(J\ideal K[X_1,\dots, X_k, t]\) be a \(\deg_{\mathbf w}\)\=/homogeneous ideal and let \(g_1,\dots, g_s\in K[X_1,\dots, X_k, t]\) be a standard basis of \(S^{-1}J\) with respect to the monomial ordering \(>_M\).
  Write \(g_i = t^{m_i}g'_i\) for \(g'_i\in K[X_1,\dots, X_k, t]\) with \(t\nmid g'_i\).
  Then \(g'_1,\dots, g'_s\) generate \(J:\langle t\rangle^\infty\).
\end{proposition}
\begin{proof}
  By Lemma~\ref{lem:extconteq}, we have \(g_i\in J\) and hence \(g'_i\in J:\langle t\rangle^\infty\).

  Let \(g'\in J:\langle t\rangle^\infty\).
  Then there is \(m\geq 0\) with \(g \coloneqq t^mg'\in J\), hence \(g\in S^{-1}J\).
  So there is \(i\in\{1,\dots, s\}\) such that \(\LT_{>_M}(g_i) \mid \LT_{>_M}(g)\) and therefore \[\LT_{>_M}(g'_i) \mid t^m\LT_{>_M}(g').\]
  We may assume that the \(g_i\) are \(\deg_{\mathbf w}\)\=/homogeneous, as replacing \(g_i\) by their homogeneous parts does not change the standard basis property.
  Therefore, \(t\nmid \LT_{>_M}(g'_i)\) by choice of \(g'_i\) and the properties of \(>_M\).
  So, \(\LT_{>_M}(g'_i)\mid \LT_{>_M}(g')\) and this proves that \(g'_1,\dots,g'_s\) is a standard basis of \(S^{-1}\big(J:\langle t\rangle^\infty\big)\), so in particular a generating system \cite[Lemma~1.6.7~(3)]{GP08}.
  But then \(g'_1,\dots, g'_s\) generate \(J:\langle t\rangle^\infty\) by Lemma~\ref{lem:extconteq} again.
\end{proof}

In conclusion, to compute \(I^\hom = \tilde I : \langle t\rangle^\infty\) we need to compute a standard basis for \(\tilde I\) with respect to \(>_M\) and divide the elements by \(t\).
This only involves the computation of one standard basis and proved to be quite efficient in practice compared with the computation of the saturation via iterated quotients.
To compute the homogenization with respect to several gradings \(\deg_1,\dots,\deg_m\) as in Algorithm~\ref{alg:muvak}, we may use this approach iteratively.

\printbibliography

\end{document}